\documentclass[11pt]{article}

\usepackage{graphicx}

\usepackage{bbm}

\usepackage{amsmath}
\usepackage{amsthm}
\usepackage{amsfonts}
\usepackage{amssymb}
\usepackage{mathtools}
\usepackage{xcolor}
\RequirePackage[colorlinks,citecolor=blue,urlcolor=blue]{hyperref}

\newcommand{\eee}{{\rm e}}

\newcommand{\Erw}{\mathbb{E}}
\newcommand{\N}{\mathbb{N}}
\newcommand{\R}{\mathbb{R}}
\newcommand{\C}{\mathbb{C}}
\DeclareMathOperator{\1}{\mathbbm{1}}

\newcommand{\wh}{\widehat}

\newcommand{\ovl}{\overline}

\newcommand{\cN}{\mathcal{N}}
\newcommand{\Prob}{\mathbb{P}}

\newcommand{\eps}{\varepsilon}

\newtheorem{thm}{Theorem}[section]
\newtheorem{lemma}[thm]{Lemma}

\newtheorem{cor}[thm]{Corollary}

\newtheorem{assertion}[thm]{Proposition}
\theoremstyle{definition}

\theoremstyle{remark}
\newtheorem{rem}[thm]{Remark}

\begin{document}
\title{On decoupled standard random walks}\date{}
\author{Gerold Alsmeyer\footnote{Department of Mathematics and Computer Science, University of M\"{u}nster, Germany; e-mail address: gerolda@uni-muenster.de} \ \ Alexander Iksanov\footnote{Faculty of Computer Science and Cybernetics, Taras Shevchenko National University of Kyiv, Ukraine; e-mail address:
iksan@univ.kiev.ua} \ \ and \ \ Zakhar Kabluchko\footnote{Department of Mathematics and Computer Science, University of M\"{u}nster, Germany; e-mail address: zakhar.kabluchko@uni-muenster.de }}
\maketitle

\begin{abstract}
Let $S_{n}=\sum_{k=1}^{n}\xi_{k}$, $n\in\N$, be a standard random walk with i.i.d.~nonnegative increments $\xi_{1},\xi_{2},\ldots$ and associated renewal counting process $N(t)=\sum_{n\ge 1}\1_{\{S_{n}\le t\}}$, $t\ge 0$. A  decoupling of $(S_{n})_{n\ge 1}$ is any sequence $\wh{S}_{1}$, $\wh{S}_{2},\ldots$ of independent random variables such that, for each $n\in\N$, $\wh{S}_{n}$ and $S_{n}$ have the same law. Under the assumption that the law of $\wh{S}_{1}$ belongs to the domain of attraction of a stable law
with finite mean, we prove a functional limit theorem for the \emph{decoupled renewal counting process} $\wh{N}(t)=\sum_{n\ge 1}\1_{\{\wh{S}_{n}\le t\}}$, $t\ge 0$, after proper scaling, centering and normalization. We also study the asymptotics of $\log \Prob\{\min_{n\ge 1}\wh{S}_{n}>t\}$ as $t\to\infty$ under varying assumptions~on~the law of $\wh{S}_{1}$. In particular, we recover the assertions which were previously known in the case when $\wh{S}_{1}$ has an exponential law. These results, which were formulated in terms of an infinite Ginibre point process, served as an initial motivation for the present work. Finally, we prove strong law of large numbers type results for the sequence of decoupled maxima $M_{n}=\max_{1\le k\le n}\wh{S}_{k}$, $n\in\N$, and the related first passage time process $\wh\tau(t)=\inf\{n\in\N: M_{n}>t\}$, $t\ge 0$. In particular, we provide a tail condition on the law of $\wh{S}_{1}$ in the case when the latter has finite mean but infinite variance that implies $\lim_{t\to\infty}t^{-1}\wh\tau(t)=\lim_{t\to\infty}t^{-1}\Erw\wh\tau(t)=0$. In other words, $t^{-1}\wh\tau(t)$ may exhibit a different limit behavior than $t^{-1}\tau(t)$, where $\tau(t)$ denotes the level-$t$ first passage time of $(S_{n})_{n\ge 1}$.
\end{abstract}

\noindent Key words: decoupled renewal process; functional limit theorem; large deviation; renewal theory; stationary Gaussian process; strong law of large numbers; tail behavior  

\noindent 2020 Mathematics Subject Classification: Primary: 60F15, 60F17
\hphantom{2020 Mathematics Subject Classification: } Secondary: 60F10, 60G15

\section{Introduction and main results}

For a given sequence $(\xi_{n})_{n\ge 1}$ of i.i.d.~nonnegative random variables, consider the associated 
standard random walk $S_{n}=\sum_{k=1}^{n}\xi_{k}$ for $n\ge 1$. Further, let $(N(t))_{t\ge 0}$ and $(\tau(t))_{t\ge 0}$ denote the associated \emph{renewal counting process} and \emph{first-passage time process}, respectively, which are defined by
$$ N(t)\,:=\,\sum_{n\ge 1}\1_{\{S_{n}\le t\}}\quad\text{and}\quad\tau(t)\,:=\,N(t)+1=\inf\{n\ge 1:S_{n}>t\} $$
for $t\ge 0$. In this article, we are interested in decoupled versions of these processes, which are obtained by replacing $(S_{n})_{n\ge 1}$ with a \emph{decoupling} $(\wh{S}_{n})_{n\ge 1}$, i.e., with a sequence of independent $\wh{S}_{1},\wh{S}_{2},\ldots$ such that $\wh{S}_{n}$ is a copy of $S_{n}$ for each $n\in\N$. The counterparts of $N(t)$ and $\tau(t)$ for this decoupling are denoted $\wh{N}(t)$ and $\wh{\tau}(t)$ and we note that, with $M_{n}:=\max_{1\le k\le n}\wh{S}_{k}$ for $n\in\N$,
$$ \wh{\tau}(t)\,=\,\inf\{n\ge 1:M_{n}>t\},\quad t\ge 0. $$
As $\lim_{n\to\infty}M_{n}=+\infty$ a.s., we have $\wh{\tau}(t)<\infty$ a.s.~for all $t$.

\vspace{.1cm}
Our interest in the objects just introduced was raised by their recent appearance in connection with particular determinantal point processes. To be more precise, let $\C$ as usual denote the set of complex numbers, $\bar z$ the complex conjugate of $z\in\C$, ${\rm Leb}$ 
Lebesgue~measure on $\C$, and finally $\rho$ the measure defined by $\rho({\rm d}z):=\pi^{-1}\eee^{-|z|^{2}}{\rm Leb}({\rm d}z)$ for $z\in\C$. Then $\Theta$ is called an \emph{infinite Ginibre point process} on $\C$ if it is a determinantal point process with kernel $C(z,w)=\eee^{z\bar w}$ for $z,w\in \C$ with respect to $\rho$, which in turn means that $\Theta$ is a simple point process such that, for any $k\in\N$ and any pairwise disjoint Borel subsets $B_{1},\ldots, B_{k}$ of $\C$,
$$ \Erw\prod_{j=1}^{k}\Theta(B_{j})\ =\ \int_{B_{1}\times\ldots\times B_{k}} {\rm det}(C(z_{i}, z_{j}))_{1\le i,j\le k}\ \rho({\rm d}z_{1})\ldots\rho({\rm d}z_{k}). $$
See \cite{Hough+Krishnapur+Peres+Virag:2009} for detailed information on determinantal point processes, in particular Sections 4.3.7 and 4.7 for a discussion of the Ginibre point process.

\vspace{.1cm}
For $t\ge 0$, let $\Theta(D_t)$ denote the number of points of $\Theta$ in the disk $D_t:=\{z\in \C: |z|<t^{1/2}\}$. According to an infinite version of Kostlan's result \cite{Kostlan:1992}, stated as Theorem 1.1 in \cite{Fenzl+Lambert:2021}, the process
\begin{equation}\label{eq:distr}
(\Theta(D_t))_{t\ge 0}~~\text{has the same law as}~~(\wh{N}(t))_{t\ge 0}\ =\ \Big(\textstyle\sum_{n\ge 1}\1_{\{\wh{S}_{n}\le t\}}\Big)_{\!t\ge 0},
\end{equation}
where $\xi$ is a standard exponential random variable and thus $(\wh{N}(t))_{t\ge 0}$ a \emph{decoupled standard Poisson process}. Prop.~1.4 in \cite{Fenzl+Lambert:2021} is a functional limit theorem for $(\Theta(D_t))_{t\ge 0}$, properly scaled,~centered and normalized. Prop.~7.2.1 on p.~124 in \cite{Hough+Krishnapur+Peres+Virag:2009} provides the first-order asymptotics of the logarithmic hole probability for an infinite Ginibre point process. When formulated in terms of a decoupled Poisson process and thus assuming the law of $\xi$ to be standard exponential, this is equivalent to the first-order asymptotics of $\log\Prob\{\min_{n\ge 1}\wh{S}_{n}>t\}$ as $t\to\infty$. The main purpose of the present paper is to prove corresponding results for $(\wh{N}(t))_{t\ge 0}$ and $\log\Prob\{\min_{n\ge 1}\wh{S}_{n}>t\}$ as $t\to\infty$ \emph{without specifying the law of $\xi$}. Additionally, we provide  strong law of large numbers type results for $(M_{n})_{n\ge 1}$ and $(\wh{\tau}(t))_{t\ge 0}$ and also find the first-order asymptotics of $\Erw\wh{\tau}(t)$ as $t\to\infty$. Functional limit theorems for $(M_{n})_{n\ge 1}$ and $(\wh{\tau}(t))_{t\ge 0}$, however, will be discussed in a separate article.

\section{Weak convergence of the decoupled renewal counting process}\label{sect:weak convergence}

We will state our functional limit theorem for $(\wh{N}(t))_{t\ge 0}$ in Subsection \ref{main result} below after a brief review of corresponding results for suitable normalizations of $(N(t))_{t\ge 0}$ and $(\tau(t))_{t\ge 0}$ which are known in the literature. Our result will assume that the law of $\xi$ belongs to the domain of attraction of a stable law with index $\alpha\in (1,2]$. This particularly entails $\mu:=\Erw\xi<\infty$. As for $\alpha=2$, let us recall that the law of $\xi$ belongs to the domain of attraction of a normal distribution if, and only if, either $\sigma^{2} :={\rm Var}\, \xi\in (0, \infty)$, or ${\rm Var}\, \xi=\infty$ and the truncated mean of $\xi^{2}$ is slowly varying at infinity, thus
\begin{equation}\label{eq:slowly varying 2nd moment}
\Erw\xi^{2} \1_{\{\xi \le t\}}\,\sim\,\ell(t)\quad\text{as }t\to\infty
\end{equation}
for some slowly varying function $\ell$. And if $\alpha \in (1,2)$, then the distribution of $\xi$ belongs to the domain of attraction of an $\alpha$-stable law if, and only if,
\begin{equation}\label{eq:domain alpha}
\Prob\{\xi>t\}\,\sim\,t^{-\alpha}\ell(t)\quad\text{as }t\to\infty
\end{equation}
for some $\ell$ as before.

\subsection{A quick review of the ordinary renewal case}\label{subsect:reminder}

Let $D$ denote the Skorokhod space of c\`{a}dl\`{a}g functions defined on $[0,\infty)$. According to Theorem 5.3.1 and Theorem 5.3.2 in
\cite{Gut:2009} or Section 7.3.1 in \cite{Whitt:2002}
\begin{equation}\label{eq:FLT for N(t)}
\bigg(\frac{\tau(ut)-\mu^{-1}ut}{\mu^{-1-1/\alpha}c_{\alpha}(t)}\bigg)_{u\ge 0} ~\Longrightarrow~
S_{\alpha}:=(\mathcal{S}_{\alpha}(u))_{u\ge 0}\quad\text{as }t\to\infty,
\end{equation}
where,
\begin{itemize}
\item [(A1)]
if $\sigma^{2}<\infty$, then $\alpha=2$, $c_{2}(t)=\sigma\sqrt{t}$, $\mathcal{S}_{2}$ is standard Brownian motion, and the
convergence takes place in the $J_{1}$-topology on $D$;
\item [(A2)]
if $\sigma^{2}=\infty$ and \eqref{eq:slowly varying 2nd moment}
holds, then $\alpha=2$, $c_{2}$ is some positive continuous function such that
\begin{equation*}
\lim_{t\to\infty} t \ell(c_{2}(t))(c_{2}(t))^{-2}\,=\,1,
\end{equation*}
and the convergence takes place in the $J_{1}$- topology on $D$;
\item [(A3)] if \eqref{eq:domain alpha} holds for $\alpha\in (1,2)$, then $\mathcal{S}_{\alpha}$ is a spectrally negative $\alpha$-stable L\'{e}vy process such that
$\mathcal{S}_{\alpha}(1)$ has the characteristic function
\begin{equation}\label{eq:stable}
\mathbb{E} [\exp(iz\mathcal{S}_{\alpha}(1))]\ =\ \exp\{-|z|^\alpha
\Gamma(1-\alpha)(\cos(\pi\alpha/2)+{\rm i}\sin(\pi\alpha/2){\rm sign}\,(z))\},\quad z\in\R.
\end{equation}
Here $\Gamma$ denotes Euler's gamma function, $c_{\alpha}$ is
some positive continuous function satisfying
\begin{equation*}
\lim_{t \to \infty} t\ell(c_{\alpha}(t))(c_{\alpha}(t))^{-\alpha}\,=\,1,
\end{equation*}
and the convergence takes place in the $M_{1}$-topology on $D$.
\end{itemize}
Observe that \eqref{eq:FLT for N(t)} also holds with $N(ut)$ replacing $\tau(ut)$ and $\Erw\tau(ut)$ replacing $\mu^{-1}ut$. We refer to \cite{Whitt:2002} for extensive information concerning both the $J_{1}$- and $M_{1}$-convergence on $D$.

\vspace{.1cm}
The function $c_{\alpha}$ is regularly varying at $\infty$ with index $1/\alpha$. Hence, the function $t\mapsto t/c_{\alpha}(t)$ is regularly varying at $\infty$ with index $1-1/\alpha$. By Theorem 1.8.2 in \cite{Bingham+Goldie+Teugels:1989}, there exists an eventually strictly increasing and differentiable function $d_{\alpha}$ satisfying $\lim_{t\to\infty}(td_{\alpha}(t)/d_{\alpha}'(t))=1-1/\alpha$ and $t/c_{\alpha}(t)\sim d_{\alpha}(t)$ as $t\to\infty$. Thus, we can and do assume without loss of generality that the function $t\mapsto t/c_{\alpha}(t)$ itself possesses all these properties. With this at hand, we can put $h_{\alpha}(t):=(t/c_{\alpha}(t))^{-1}$ (inverse function) for large $t$ and point out that $h_{\alpha}$ is ultimately strictly increasing, regularly varying with index $(1-1/\alpha)^{-1}$, and
\begin{equation}\label{eq:diff}
\lim_{t\to\infty}\frac{th'_{\alpha}(t)}{h_{\alpha}(t)}\ =\ \lim_{t\to\infty}\frac{h_{\alpha}'(t)}{c_{\alpha}(h_{\alpha}(t))}\ =\ \frac{\alpha}{\alpha-1}.
\end{equation}

\subsection{Functional limit theorem for the decoupled renewal counting process}\label{main result}

Denote by $D(I)$ the Skorokhod space of c\`{a}dl\`{a}g functions defined on an interval $I$, by $\cN(0,1)$ the standard normal law and by $\Phi$ its distribution function. We write ${\overset{{\rm f.d}}\longrightarrow}$  and ${\overset{{\rm d}}\longrightarrow}$ for weak convergence of finite-dimensional and one-dimensional distributions, respectively, and in the statement of Theorem \ref{thm:main}, the random variable $\mathcal{S}_{\alpha}(1)$, $h_{\alpha}$ and a smooth version of $c_{\alpha}$ are as defined in the previous subsection. Finally, let $V$ be the renewal function of $(S_{n})_{n\ge 1}$ and thus also its decoupling $(\wh{S}_{n})_{n\ge 1}$, that is
$$ V(t)\,:=\,\sum_{n\ge 1}\Prob\{S_{n}\le t\}\,=\,\sum_{n\ge 1}\Prob\{\wh{S}_{n}\le t\}\,=\,\Erw\wh{N}(t),\quad t\ge 0. $$

\begin{thm}\label{thm:main}
If (A1), (A2), or (A3) holds, then
$$ \bigg(\frac{\wh{N}(h_{\alpha}(t+u))-V(h_{\alpha}(t+u)) 
}{(\mu^{-1-1/\alpha} c_{\alpha}(h_{\alpha}(t)))^{1/2}}\bigg)_{u\in\R}\ {\overset{{\rm f.d}}\longrightarrow}\ X_{\alpha}\quad\text{as }t\to\infty, $$
where $X_{\alpha}=(X_{\alpha}(u))_{u\in\R}$ is a centered stationary Gaussian process with covariance function
\begin{equation}\label{eq:covar1}
{\rm Cov}\,(X_{\alpha}(u),X_{\alpha}(v))\ =\ \int_{\R} \Prob\{\mathcal{S}_{\alpha}(1)>a_{\alpha} (u\vee v)+y\}\Prob\{\mathcal{S}_{\alpha}(1)\le a_{\alpha}(u\wedge v)+y\}\ {\rm d}y
\end{equation}
for $u,v\in\R$ and $a_{\alpha}:=\mu^{1/\alpha}\alpha/(\alpha-1)$. Furthermore,
\begin{align}\label{eq:covar}
{\rm Cov}\,(X_{2}(u), X_{2}(v))\ =\ \pi^{-1/2}\exp(-a_{2}^{2}(u-v)^{2}/4)-a_{2} |u-v|\big(1-\Phi(2^{-1/2}a_{2}|u-v|)\big)
\end{align}
for all $u,v\in\R$. Under the additional assumption that the function $V$ is Lipschitz continuous on $[0,\infty)$, even
$$ \bigg(\frac{\wh{N}(h_{\alpha}(t+u))-V(h_{\alpha}(t+u)) 
}{(\mu^{-1-1/\alpha} c_{\alpha}(h_{\alpha}(t)))^{1/2}}\bigg)_{u\in\R}\ \Longrightarrow\ X_{\alpha}\quad\text{as } t\to\infty $$
in the $J_{1}$-topology on $D(\R)$ holds true. 
\end{thm}
\begin{rem}
If (A1) holds, in particular $\alpha=2$, and $V$ is Lipschitz continuous, then $h_{2}(t)=\sigma^{2}t^{2}$ and $c_{2}(h_{2}(t))=\sigma^{2} t$ for all $t>0$. As a consequence, the limit assertion of Theorem \ref{thm:main} takes the simpler form
\begin{equation}\label{eq:A11}
\bigg(\frac{\wh{N}(\sigma^{2} (t+u)^{2})- V(\sigma^{2} (t+u)^{2})}{(\mu^{-3/2}\sigma^{2} t)^{1/2}}\bigg)_{u\in\R}~\Longrightarrow~X_{2}\quad\text{as }t\to\infty,
\end{equation}
and $a_{2}=2\mu^{1/2}$. Standard renewal theory provides
$$ -1\,\le\,V(t)-\mu^{-1}t\,\le\,\mu^{-2}\Erw\xi^{2}-1\quad\text{for all }t\ge 0, $$
the left-hand side being a consequence of $t\le \Erw S_{\tau(t)}=\mu\Erw\tau(t)=\mu (V(t)+1)$ for $t\ge 0$ (using Wald's identity), the right-hand side of Lorden's inequality. Hence, by replacing $V(\sigma^{2} (t+u)^{2})$ with $\mu^{-1}\sigma^{2} (t+u)^{2}$ in \eqref{eq:A11}, we conclude that\footnote{For $\alpha>1$ and close to $1$, $(V(t)-\mu^{-1}t)/c_{\alpha}(t)^{1/2}$ does not converge to $0$ as $t\to\infty$. Whenever this is the case, the centering in Theorem \ref{thm:main} cannot be replaced with $\mu^{-1}h_{\alpha}(t+u)$.}
\begin{equation}\label{eq:A1}
\bigg(\frac{\wh{N}(\sigma^{2} (t+u)^{2})-\mu^{-1}\sigma^{2} (t+u)^{2}}{(\mu^{-3/2}\sigma^{2} t)^{1/2}}\bigg)_{u\in\R}~\Longrightarrow~X_{2}\quad\text{as }t\to\infty.
\end{equation}
Assuming the law of $\xi$ to be standard exponential (thus $\mu=\sigma^{2}=1$ and $V(t)=t$ for $t\ge 0$), we recover the result obtained in \cite{Fenzl+Lambert:2021} as Proposition 1.4 and Remark 1 on p.~7424. Putting $u=1$ in \eqref{eq:A1} and noting that, by \eqref{eq:covar}, ${\rm Var}\,X_{2}(1)=\pi^{-1/2}$, we obtain a one-dimensional central limit theorem $$\frac{\wh{N}(t)-\mu^{-1}t}{(\pi^{-1}t)^{1/4}}\ {\overset{{\rm d}}\longrightarrow}\ \cN(0,1).$$
\end{rem}
\begin{rem}
For $\alpha\in (1,2)$, it seems that ${\rm Cov}\,(X_{\alpha}(u), X_{\alpha}(v))$ does not admit a useful semi-explicit representation like \eqref{eq:covar}. However, according to Lemma \ref{lem:calcintegr} below $${\rm Var}\,X_{\alpha}(u)=\pi^{-1}\Gamma(1-1/\alpha)(2\Gamma(1-\alpha)\cos(\pi\alpha/2))^{1/\alpha},\quad u\in\R,$$ where $\Gamma$ is Euler's gamma function.
\end{rem}

Let $W$ be Gaussian white noise on $\R\times [0,1]$ with intensity measure being Lebesgue measure ${\rm Leb}$. This means that, for any Borel sets $A, B\subseteq \R\times [0,1]$ of finite Lebesgue measure, $W(A)$ is a zero-mean Gaussian random variable and
$\Erw W(A)W(B) ={\rm Leb}(A\cap B)$. The weak limit $X_{\alpha}$ arising in Theorem \ref{thm:main} admits an integral representation with respect to $W$. 
\begin{thm}\label{thm:integral}
Putting $\Phi_{\alpha}(y):=\Prob\{\mathcal{S}_{\alpha}(1)\le y\}$ for $y\in\R$, the process $Y_{\alpha}:=(Y_{\alpha}(u))_{u\in\R}$ defined by
$$ Y_{\alpha}(u)\ :=\ \int_{\R\times [0,1]}(\1_{\{y\le \Phi_{\alpha}(a_{\alpha} u+x)\}}-\,\Phi_{\alpha}(a_{\alpha} u+x))\ W({\rm d}x, {\rm d}y)\quad\text{for } u\in\R $$
is a stationary centered Gaussian process with the same covariance function as $X_{\alpha}$, so
$$ {\rm Cov}\,(Y_{\alpha}(u),Y_{\alpha}(v))\ =\ {\rm Cov}\,(X_{\alpha}(u),X_{\alpha}(v))\quad\text{for all }u,v\in\R. $$
Moreover, $Y_{\alpha}$ has a version with sample paths which are H\"{o}lder continuous with exponent $\gamma$ for any $\gamma\in (0,1/2)$.
\end{thm}

\section{Tail asymptotics for the minimum of the decoupling $(\wh{S}_{n})_{n\ge 1}$}\label{sect:minimum}

In this section, we focus on the logarithmic asymptotics of
$$ \Prob\Big\{\min_{n\ge 1}\,\wh{S}_{n}>t\Big\}\ =\ \prod_{n\ge 1}\Prob\{\wh{S}_{n}>t\}\ =\ \prod_{n\ge 1}\Prob\{S_{n}>t\} $$
as $t\to\infty$ under various assumptions on the distribution of $\xi$. Subsection \ref{sect:light} treats the case when the law of $\xi$ has light tails, that is, when $\Erw\exp(s_{0}\xi)<\infty$ for some $s_{0}>0$, whereas Subsection \ref{sect:heavy} is devoted to the case when the law of $\xi$ has heavy tails and thus $\Erw\exp(s\xi)=\infty$ holds for all $s>0$.

\subsection{Light tails}\label{sect:light}

Under mild assumptions including $\mu=\Erw\xi<\infty$, we will show in Lemma \ref{lem:aux} that the variables $\wh{S}_{n}$ for $n>\lfloor t/\mu\rfloor$ do not contribute to the logarithmic asymptotics of $\Prob\{\min_{n\ge 1}\,\wh{S}_{n}>t\}$ as $t\to\infty$. Under the assumptions of Theorem \ref{thm:min}(a), these asymptotics are driven by $\wh{S}_{n}$ for $n\in [\lfloor at\rfloor, \lfloor t/\mu\rfloor]$ and positive $a$ close to 0. They are therefore determined by the large deviations of the standard random walk $(S_{n})_{n\ge 1}$, which in turn are described by Cram\'er's theorem. This particularly explains the appearance of the Legendre transform $I$ in part (a). Under the assumptions of~Theorem \ref{thm:min}(b2), the asymptotics are driven by the first elements of the sequence $(\wh{S}_{n})_{n\in\N}$ and are thus determined by $-\log \Prob\{\xi>t\}$ as $t\to\infty$. The setting treated in part (b1) is intermediate between the aforementioned two, which manifests itself in $-\log \Prob\{\xi>t\}\sim I(t)$ as $t\to\infty$.

\begin{thm}\label{thm:min}
(a) Assume that
\begin{gather}\label{eq:s0}
\Erw \eee^{s_0\xi}\,<\,\infty\quad\text{for some }s_{0}>0
\shortintertext{and}
\label{eq:integral}
\int_{0}^{1} -y\log \Prob\{\xi>1/y\}\ {\rm d}y\ <\ \infty.
\shortintertext{Then}
\lim_{t\to\infty}-t^{-2}\log \Prob\Big\{\min_{n\ge 1}\,\wh{S}_{n}>t\Big\}\ =\ \int_{0}^{1/\mu}yI(1/y){\rm d}y\ <\ \infty,
\end{gather}
where $\mu=\Erw\xi<\infty$,  $I$ denotes the Legendre transform of the distribution of $\xi$, that is,
$$ I(x)\,:=\,\sup_{s\in J}\,(sx-\log \Erw\exp(s\xi))\quad\text{for }x>0, $$
and $J:=\{s\ge 0: \Erw\exp(s\xi)<\infty\}$.

\vspace{.2cm}\noindent
(b) For some $\alpha\ge 2$ and some $\ell$ slowly varying at $\infty$, assume
\begin{equation}\label{eq:log tail condition}
\lim_{t\to\infty}\frac{-\log \Prob\{\xi>t\}}{t^\alpha\ell(t)}\,=\,c\in (0,\infty).
\end{equation}

\noindent
(b1) If $\alpha=2$, \eqref{eq:integral} fails to hold, and
\begin{equation}\label{eq:bojanic}
\lim_{t\to\infty}\bigg(\frac{\ell(\lambda t)}{\ell(t)}-1\bigg)\log \ell(t)\,=\,0\quad\text{for some }\lambda>1,
\end{equation}
then
$$ \lim_{t\to\infty}\frac{-\log \Prob\{\min_{n\ge 1}\,\wh{S}_{n}>t\}}{t^{2}\ell^{*}(t)}\ =\ c, $$
where $\ell^{*}(t):=\int_{1}^t y^{-1}\ell(y){\rm d}y$ satisfies $\lim_{t\to\infty}\ell^{*}(t)=\infty$.

\vspace{.2cm}\noindent
(b2) If $\alpha>2$, then $$\lim_{t\to\infty}\frac{-\log \Prob\{\min_{n\ge 1}\,\wh{S}_{n}>t\}}{t^\alpha\ell(t)}\ =\ c\zeta(\alpha-1),$$ where $\zeta(x)=\sum_{n\ge 1}n^{-x}$ for $x>1$ is the Riemann zeta function.
\end{thm}
\begin{rem}
We stress that Condition \eqref{eq:integral} does not necessarily entail \eqref{eq:s0}. For instance, if $-\log \Prob\{\xi>t\}\sim ct^\alpha$ for some $c>0$ and $\alpha\in (0,1)$, then \eqref{eq:integral} holds, but \eqref{eq:s0} does not. A suf\-ficient condition for both \eqref{eq:s0} and \eqref{eq:integral} is $-\log \Prob\{\xi>t\}\sim t^\alpha \ell(t)$ for some $\alpha\in [1,2)$ and some $\ell$ slowly varying at $\infty$. If $\alpha=2$, then \eqref{eq:integral} holds for some $\ell$ and fails to hold for the other.
\end{rem}
\begin{rem}
Let $\xi$ have a standard exponential distribution ($\mu=1$). Then $I(x)=x-1-\log x$ for $x>0$ and $\int_{0}^{1}yI(1/y){\rm d}y=\int_{0}^{1}(1-y+y\log y){\rm d}y=1/4$. With this at hand, we recover the result obtained in Proposition 7.2.1 on p.~124 of \cite{Hough+Krishnapur+Peres+Virag:2009}.
\end{rem}
\begin{rem}
Relation \eqref{eq:bojanic} is satisfied if $\ell$ converges to a positive constant, by $\ell(x)=(\log_{k} x)^\alpha$ for $\alpha\in\R$, where $\log_{k}$ is the $k$th iterate of $\log$, and by products of such $\ell$, see Example 1 on p.~433 in \cite{Bingham+Goldie+Teugels:1989}.
\end{rem}

\subsection{Heavy tails}\label{sect:heavy}

The case when the law of $\xi$ has heavy tails is divided into two subcases treated in Theorems \ref{thm:heavy} and \ref{thm:semi}. In the first subcase, the law of $\xi$ has regularly varying tails of index $0<\alpha\ne 1$. Then,
\begin{itemize}
\item if $\alpha\in (0,1)$ and thus $\mu=\Erw\xi=\infty$, the logarithmic asymptotics of $\Prob\{\min_{n\ge 1}\,\wh{S}_{n}>t\}$ are driven by the variables $\wh{S}_{n}$ for $n\in [\lfloor a/\Prob\{\xi>t\}\rfloor,\lfloor b/\Prob\{\xi>t\}\rfloor]$ with positive $a$ close to $0$ and large $b$ (Theorem \ref{thm:heavy}(a)) and thus by the distributional convergence of $\Prob\{\xi>t\}\tau(t)$. For further explanation, we refer to \eqref{eq:distrconv} where the convergence is stated.
\item if $\alpha>1$, these asymptotics are driven by the $\wh{S}_{n}$ for $n\in [\lfloor at \rfloor, \lfloor t/(\mu+\delta)\rfloor]$ with positive $a$ and $\delta$ close to $0$ and therefore by the large deviation behavior of the random walk $(S_{n})_{n\ge 1}$. More importantly, such $n$ belong to the `one big-jump domain', that is, $\Prob\{S_{n}-\mu n>t\}\sim \Prob\{\max_{1\le k\le n}\xi_{k}>t\}\sim n\Prob\{\xi>t\}$ as $t\to\infty$. (Theorem \ref{thm:heavy}(b))
\end{itemize}
In the second subcase, treated by Theorem \ref{thm:semi}, the driving force behind the asymptotics is still the `one big-jump domain', which covers all positive integers $n\le \lfloor t/\mu\rfloor$. All larger integers $n$ do not contribute to the asymptotics in question as will be shown in Lemma \ref{lem:aux}.

\vspace{.1cm}
In order to state our results, let $(W_{\alpha}(t))_{t\ge 0}$ for $\alpha\in (0,1)$ denote a drift-free $\alpha$-stable subordinator with
$$ -\log \Erw\exp(-zW_{\alpha}(t))\ =\ \Gamma(1-\alpha) tz^\alpha\quad\text{for }z\ge 0, $$
where $\Gamma$ is again Euler's gamma function. Let further $W_{\alpha}^{\leftarrow}$ denote an inverse $\alpha$-stable subordinator, defined by $W_{\alpha}^{\leftarrow}(t):=\inf\{s\ge 0: W_{\alpha}(s)>t\}$ for $t\ge 0$. The law of $W_{\alpha}^{\leftarrow}(1)$ is known in the literature as a Mittag-Leffler distribution with parameter $\alpha\in (0,1)$, the name stemming from the fact that
\begin{equation}\label{mgf}
\Erw\exp(s\Gamma(1-\alpha)W_{\alpha}^{\leftarrow}(1))\ =\ \sum_{n\ge 0}\frac{s^n}{\Gamma(1+n\alpha)},\quad s\ge 0,
\end{equation}
and that the right-hand side defines the Mittag-Leffler function with parameter $\alpha$, a generalization of the exponential function which corresponds to $\alpha=1$.

\begin{thm}\label{thm:heavy}
Assume $\Prob\{\xi>t\}\sim t^{-\alpha}\ell(t)$ as $t\to\infty$ for some $\alpha>0$ and some $\ell$ slowly varying at $\infty$.

\vspace{.2cm}\noindent
(a) If $\alpha\in (0,1)$, then
$$ \lim_{t\to\infty} -\log \Prob\Big\{\min_{n\ge 1}\,\wh{S}_{n}>t\Big\}\,\Prob\{\xi>t\}\ =\ \int_{0}^{\infty} -\log \Prob\{W_{\alpha}^{\leftarrow}(1)\le x\}\ {\rm d}x\ <\ \infty.$$

\vspace{.2cm}\noindent
(b) If $\alpha>1$, then
$$ \lim_{t\to\infty}\frac{-\log \Prob\{\min_{n\ge 1}\,\wh{S}_{n}>t\}}{t\log t}\ =\ \frac{\alpha-1}{\mu}, $$
where $\mu=\Erw\xi<\infty$.
\end{thm}

\begin{thm}\label{thm:semi}
Assume $\Prob\{\xi>t\}=\eee^{-t^{\alpha}\ell(t)}$ for $t>0$, some $\alpha\in (0,1)$ and some $\ell$ slowly varying at $\infty$. Putting $H(t):=-\log \Prob\{\xi>t\}$, assume also
\begin{equation}\label{eq:borov}
H(t+o(t))-H(t)\ =\ \alpha\,o(t)\,t^{-1}H(t)(1+o(1))+o(1)\quad\text{as }t\to\infty.
\end{equation}
Then
$$ \lim_{t\to\infty} \frac{-\log \Prob\{\min_{n\ge 1}\,\wh{S}_{n}>t\}}{t^{\alpha+1}\ell(t)}\ =\ \frac{1}{\mu (\alpha+1)}, $$
where $\mu=\Erw\xi<\infty$.
\end{thm}

\begin{rem}
We note that \eqref{eq:borov} is not very restrictive and refer to p.~931 in \cite{Borovkov+Mogulskii:2006}, where sufficient conditions are provided.
\end{rem}

\section{The sequence of decoupled maxima and first-passage times}

Recall that $M_{n}=\max_{1\le k\le n}\,\wh{S}_{k}$ for $n\in\N$, $\wh \tau(t)=\inf\{n\in\N: M_{n}>t\}$ for $t\ge 0$ and $\mu=\Erw\xi$. We state our result in the subsequent theorem.

\begin{thm}\label{thm:SLLN maxima}
Let the law of $\xi$ be nondegenerate. Then the following assertions hold.
\begin{itemize}
\item[(a)] If $\Erw\xi^{2}<\infty$, then
\begin{equation}\label{eq:slln}
\lim_{n\to\infty}\frac{M_{n}}{n}\,=\,\mu\quad\text{and}\quad\lim_{t\to\infty}\frac{\wh \tau(t)}{t}\,=\,\frac{1}{\mu}\quad\text{a.s.}
\end{equation}
\item[(b)] If $\mu<\infty$ and $\Erw\xi^{2}=\infty$, then
\begin{gather}\label{eq:slln2}
\limsup_{n\to\infty}\frac{M_{n}}{n}\,=\,\infty\quad\text{and}\quad\liminf_{t\to\infty}\frac{\wh \tau(t)}{t}\,=\,0\quad\text{a.s.}
\shortintertext{Moreover, even}
\lim_{n\to\infty}\frac{M_{n}}{n}\,=\,\infty\quad\text{and}\quad\lim_{t\to\infty}\frac{\wh\tau(t)}{t}\,=\,0\quad\text{a.s.}\label{eq:slln3}
\intertext{holds under the additional assumption $\lim_{t\to\infty}t^{2}\,\Prob\{\xi>t\}/\log\log t=\infty$, whereas}
\liminf_{n\to\infty}\frac{M_{n}}{n}\,=\,\mu\quad\text{and}\quad\limsup_{t\to\infty}\frac{\wh \tau(t)}{t}\,=\,\frac{1}{\mu}\quad\text{a.s.}\label{eq:slln4}
\end{gather}
if $\lim_{t\to\infty}t^{2}\,\Prob\{\xi>t\}/\log\log t=0$.
\item[(c)] If $\mu=\infty$, then \eqref{eq:slln3} holds.
\item[(d)] The family $\{t^{-1}\wh{\tau}(t):t\ge t_{0}\}$ is uniformly integrable for any $t_{0}>0$ and therefore
$$ \lim_{t\to\infty}\frac{\Erw\wh\tau(t)}{t}\ =\ \lim_{t\to\infty}\frac{\wh\tau(t)}{t} $$
whenever the second limit exists a.s. In particular, the limit is equal to $1/\mu$ if $\Erw\xi^{2}<\infty$.
\end{itemize}
\end{thm}

\section{Proofs for Section \ref{sect:weak convergence}}

\subsection{Auxiliary results}

For $x\in\R$, we put as common $x_{+}$ for $\max(x,0)$ and $x_{-}=\max (-x,0)$.

\begin{lemma}\label{lem:calcintegr}
Let $\theta$ be a zero-mean random variable with distribution function $F$ and $\theta_{1},\theta_{2}$ be two independent copies. Then
\begin{equation}\label{eq:curious}
I_{a}\ :=\ \int_{\R} F(x+a)(1-F(x))\ {\rm d}x\ =\ \Erw (\theta_{1}-\theta_{2}-a)_{+}\ <\ \infty,
\end{equation}
for all $a\in\R$. Moreover,
\begin{itemize}
\item[(a)] if $\theta=\mathcal{S}_{2}(1)$ and thus has law $\cN(0,1)$, then
$$ I_{a}\,=\,\pi^{-1/2}\exp(-a^{2}/4)+a\,\Phi(2^{-1/2}a), $$
in particular $I_{0}=\pi^{-1/2}$, and \eqref{eq:covar} holds true.
\item[(b)] if $\theta=\mathcal{S}_{\alpha}(1)$ for $\alpha\in (1,2)$ and thus has a spectrally negative $\alpha$-stable law with characteristic function given by \eqref{eq:stable}, then $I_{0}=\pi^{-1}\Gamma(1-1/\alpha)(2\Gamma(1-\alpha)\cos(\pi\alpha/2))^{1/\alpha}$.
\end{itemize}
\end{lemma}
\begin{proof}
Eq.~\eqref{eq:curious} is a consequence of
\begin{align*}
\int_{\R}&F(x+a)(1-F(x))\ {\rm d}x\\
&=\ \int_{0}^{\infty} (1-F(x+a)F(x))\ {\rm d}x\,-\,\int_{0}^{\infty} (1-F(x+a))\ {\rm d}x\\
&\qquad+\ \int_{-\infty}^{0} F(x+a)\ {\rm d}x\,-\,\int_{-\infty}^{0} F(x+a)F(x)\ {\rm d}x\\
&=\ \Erw (\max (\theta_{1}-a,\theta_{2}))_{+}\,-\,\Erw (\theta-a)_{+}+\Erw (\theta-a)_{-}\,-\,\Erw (\max (\theta_{1}-a,\theta_{2}))_{-}\\
&=\ a\,+\,\Erw (\max (\theta_{1}-a,\theta_{2}))_{+}\,-\,\Erw (\max (\theta_{1}-a,\theta_{2}))_{-}\\
&=\ a\,+\,\Erw\max (\theta_{1}-a,\theta_{2})\\
&=\ a\,+\,\Erw (\max (\theta_{1}-a,\theta_{2})-\theta_{2})\\
&=\ a\,+\,\Erw (\theta_{1}-\theta_{2}-a)_{+}.
\end{align*}

(a) If $\theta$ has the standard normal law, then the law of $\theta_{1}-\theta_{2}-a$ is normal with mean $-a$ and variance $2$ and has density $x\mapsto \exp (-(x+a)^{2}/4)/(2\pi^{1/2})$. Consequently,
\begin{align*}
\Erw (\theta_{1}-\theta_{2}-a)_{+}\ &=\ \frac{1}{2\pi^{1/2}}\int_{0}^{\infty} x \exp(-(x+a)^{2}/4)\,{\rm d}x\ =\ \frac{\exp(-a^{2}/4)}{\pi^{1/2}}\,-\,a\,\Prob\{\theta>2^{-1/2}a\}.
\end{align*}
Putting $a=0$, we see that $I_{0}=\pi^{-1/2}$, and a change of variable $x=a_{2}(u\vee v)+y$ in \eqref{eq:covar1} provides
\begin{align*}
&{\rm Cov}\,(X_{2}(u), X_{2}(v))\\
&\quad=\ \int_{\R} \Prob\{\mathcal{S}_{2}(1)\le -a_{2}|u-v|+x\}-\Prob\{\mathcal{S}_{2}(1)\le -a_{2}|u-v|+x\}\Prob\{\mathcal{S}_{2}(1)\le x\}\ {\rm d}x\\
&\quad=I_{-a_{2}|u-v|}\ =\ \pi^{-1/2}\exp(-a_{2}^{2}(u-v)^{2}/4)-a_{2}|u-v|\Prob\{\theta\le -2^{-1/2}a_{2}|u-v|\}\\
&\quad=\ \pi^{-1/2}\exp(-a_{2}^{2}(u-v)^{2}/4)-a_{2}|u-v|\Prob\{\theta> 2^{-1/2}a_{2}|u-v|\},
\end{align*}
and thus validity of \eqref{eq:covar}.

\vspace{.2cm}
(b) Using
$$ |x|\ =\ \pi^{-1}\int_{\R} y^{-2}(1-\cos(xy))\ {\rm d}y\quad\text{for }x\in\R, $$
one finds $\Erw |\theta|=\pi^{-1}\int_{\R}y^{-2}(1-\Erw\exp({\rm i}y\theta))\,{\rm d}y$ and then $\Erw\theta_{+}=\pi^{-1}\int_{0}^{\infty} y^{-2}(1-\Erw\exp({\rm i}y\theta))\,{\rm d}y$ for any
random variable $\theta$ with a symmetric law. Now, if $\theta$ has the characteristic function given by \eqref{eq:stable}, then
$$ \Erw\exp({\rm i}z(\theta_{1}-\theta_{2}))\ =\ \exp(-c|z|^\alpha)\quad\text{for }z\in\R, $$
where $c=2\Gamma(1-\alpha)\cos(\pi\alpha/2)$. Moreover,
\begin{align*}
\Erw (\theta_{1}-\theta_{2})_{+}\ &=\ \frac{1}{\pi}\int_{0}^{\infty} y^{-2}(1-\exp(-cy^\alpha)\ {\rm d}y\\
&=\ \frac{1}{\pi\alpha}\int_{0}^{\infty} y^{-(1+1/\alpha)}(1-\exp(-cy)){\rm d}y\ =\ \frac{\Gamma(1-1/\alpha)c^{1/\alpha}}{\pi},
\end{align*}
where the second equality is obtained by the change of variable and the third follows with the help of integration by parts.
\end{proof}

\begin{lemma}\label{lem:covar}
If (A1), (A2), or (A3) holds, then
\begin{align*}
&\frac{{\rm Cov}\,(\wh{N}(h_{\alpha}(t+u)),\wh{N}(h_{\alpha}(t+v)))}{\mu^{-1-1/\alpha}c_{\alpha}(h_{\alpha}(t))}\\
&\hspace{3cm}\xrightarrow{t\to\infty}\ \int_{\R} \Prob\{\mathcal{S}_{\alpha}(1)>a_\alpha
(u\wedge v)+y\}\Prob\{\mathcal{S}_{\alpha}(1)\le a_\alpha (u\vee v)+y\}\ {\rm d}y
\end{align*}
for all $u,v\in\R$, where $a_{\alpha}=\mu^{1/\alpha}\alpha/(\alpha-1)$ (cf.~Thm.\,\ref{thm:main}).
\end{lemma}

\begin{proof}
Put $S_0:=0$. For $u<v$,
\begin{align*}
&{\rm Cov}\,(\wh{N}(h_{\alpha}(t+u)), \wh{N}(h_{\alpha}(t+v)))\\
&=\ \Erw\Bigg[\sum_{k\ge 1}(\1_{\{\wh{S}_{k}\le h_{\alpha}(t+u)\}}-\Prob\{\wh{S}_{k}\le h_{\alpha}(t+u)\})\sum_{j\ge 1}(\1_{\{\wh{S}_{j}\le h_{\alpha}(t+v)\}}-\Prob\{\wh{S}_{j}\le h_{\alpha}(t+v)\})\Bigg]\\
&=\ \int_{0}^{\infty}\Prob\{S_{\lfloor x\rfloor}\le h_{\alpha}(t+u)\}\Prob\{S_{\lfloor x\rfloor}>h_{\alpha}(t+v)\}\ {\rm d}x.
\end{align*}
By putting $b_{\alpha}(t):=\mu^{-1-1/\alpha} c_{\alpha}(h_{\alpha}(t))$ for our convenience, making the change of variable $x=\mu^{-1}h_{\alpha}(t+u)+b_{\alpha}(t)y$ and using the duality relation $\{S_{k}\le z\}=\{\tau(z)>k\}$ for $k\in\N$ and $z\ge 0$, we further obtain
\begin{align}
{\rm Cov}\,&(\wh{N}(h_{\alpha}(t+u)),\wh{N}(h_{\alpha}(t+v)))\nonumber\\
&=\ b_{\alpha}(t)\int_{-\mu^{1/\alpha}h_{\alpha}(t+u)/c_{\alpha}(h_{\alpha}(t))}^{\infty} \Prob\{S_{\lfloor\mu^{-1}h_{\alpha}(t+u)+b_{\alpha}(t)y\rfloor}\le h_{\alpha}(t+u)\}\nonumber\\
&\hspace{4cm}\times\Prob\{S_{\lfloor\mu^{-1}h_{\alpha}(t+u)+b_{\alpha}(t)y\rfloor}>h_{\alpha}(t+v)\}\ {\rm d}y\nonumber\\
&=\ b_{\alpha}(t)\int_{-\mu^{1/\alpha} h_{\alpha}(t+u)/c_{\alpha}(h_{\alpha}(t))}^{\infty}\Prob\{\tau(h_{\alpha}(t+u))>\lfloor\mu^{-1}h_{\alpha}(t+u)+b_{\alpha}(t)y\rfloor\}\nonumber\\
&\hspace{4cm}\times\Prob\{\tau(h_{\alpha}(t+v))\le \lfloor\mu^{-1}h_{\alpha}(t+u)+b_{\alpha}(t)y\rfloor\}\ {\rm d}y.\label{eq:integra}
\end{align}
Put $u=1$ in \eqref{eq:FLT for N(t)} to see that, for any fixed $y\in\R$,
\begin{equation}\label{eq:lim1}
\lim_{t\to\infty}\Prob\{\tau(h_{\alpha}(t+u))>\lfloor\mu^{-1}h_{\alpha}(t+u)+b_{\alpha}(t)y\rfloor\}\ =\ \Prob\{\mathcal{S}_{\alpha}(1)>y\}.
\end{equation}
By recalling the fact that $h_{\alpha}(t)/(tc_{\alpha}(h_{\alpha}(t)))=1$ for large $t$ and combining it with the mean value theorem for differentiable functions and \eqref{eq:diff}, we obtain for some $\zeta\in [v,u]$
$$ \frac{h_{\alpha}(t+u)-h_{\alpha}(t+v)}{c_{\alpha}(h_{\alpha}(t))}\ =\ \frac{(u-v)th_{\alpha}'(t+\zeta)}{h_{\alpha}(t)}\frac{h_{\alpha}(t)}{tc_{\alpha}(h_{\alpha}(t))}~\xrightarrow{t\to\infty}~\frac{\alpha}{\alpha-1}(u-v).$$
For any fixed $y\in\R$, this entails
\begin{align}
\begin{split}\label{eq:lim2}
\lim_{t\to\infty}\Prob\{\tau(h_{\alpha}(t+v))&\le\lfloor\mu^{-1}h_{\alpha}(t+u)+b_{\alpha}(t)y\rfloor\}\\
&=\ \Prob\{\mathcal{S}_{\alpha}(1)\le\mu^{1/\alpha}\alpha(u-v)/(\alpha-1)+y\}.
\end{split}
\end{align}

We have just shown the convergence of the integrand in \eqref{eq:integra} and intend to prove next that the integral in \eqref{eq:integra} converges as well. Fixing any $r>0$, this integral taken over $[-r,r]$ plainly converges to $\int_{-r}^{r} \Prob\{\mathcal{S}_{\alpha}(1)>y\}\Prob\{\mathcal{S}_{\alpha}(1)\le \mu^{1/\alpha}\alpha(u-v)/(\alpha-1)+y\}\,{\rm d}y$
as $t\to\infty$ by dominated convergence. Moreover, for $t\ge t_{1}$, $t_{1}$ sufficiently large, and $y>r$, the integrand in \eqref{eq:integra} can be bounded from above with the help of Markov's inequality by
\begin{gather*}
\Erw\bigg(\frac{|\tau(h_{\alpha}(t+u))-\mu^{-1}h_{\alpha}(t+u)+1|}{b_{\alpha}(t+u)}\bigg)^{p} \bigg(\frac{b_{\alpha}(t+u)}{b_{\alpha}(t)}\bigg)^{p}\frac{1}{y^{p}}\\
\le\ A(\alpha, p)\sup_{t\ge t_{0}}\Erw\bigg(\frac{|\tau(t)-\mu^{-1}t+1|}{c_{\alpha}(t)}\bigg)^{p} \frac{1}{y^{p}}
\end{gather*}
for appropriate $t_{0}>0$, a positive constant $A(\alpha,p)$, and with $p=3/2$ under (A1) or (A2), and $p\in (1,\alpha)$ under (A3). Since the last supremum is finite by Theorems 1.1 and 1.2 in \cite{Iksanov+Marynych+Meiners:2016}, we have thus found an integrable bound for $y>b$. By a completely analogous argument, we obtain for $t\ge t_{2}$, $t_{2}$ sufficiently large, and $y<-r$ the integrable majorant
$$ B(\alpha, p)\sup_{t\ge t_{1}}\Erw\bigg(\frac{|\tau(t)-\mu^{-1}t+1|}{c_{\alpha}(t)}\bigg)^{p} \frac{1}{|y|^{p}} $$
with a positive constant $B(\alpha,p)$ and $p$ as before. Hence, by another appeal to the dominated convergence theorem, the integral in \eqref{eq:integra}, now taken over $(-\infty,-r)\cup (r,\infty)$, converges to $\int_{|y|>r}\Prob\{\mathcal{S}_{\alpha}(1)>y\}\Prob\{\mathcal{S}_{\alpha}(1)\le \mu^{1/\alpha}\alpha(u-v)/(\alpha-1)+y\}\,{\rm d}y$ as $t\to\infty$.
\end{proof}

\begin{cor}
The variance of $\wh{N}(t)$ exhibits the following asymptotics as $t\to\infty$:
\begin{gather*}
{\rm Var}\,\wh{N}(t)~\sim~\Big(\frac{\sigma^{2}t}{\mu^{3}\pi}\Big)^{1/2}\quad\text{under (A1)},\\
{\rm Var}\,\wh{N}(t)~\sim~\Big(\frac{1}{\mu^3\pi}\Big)^{1/2}c_{2}(t)\quad\text{under (A2)},\\
{\rm Var}\,\wh{N}(t)~\sim~\frac{\Gamma(1-1/\alpha)(2\Gamma(1-\alpha)\cos(\pi\alpha/2))^{1/\alpha}}{\mu^{1+1/\alpha}\pi}\,c_{\alpha}(t)\quad\text{under (A3)}.
\end{gather*}
\end{cor}
\begin{proof}
Lemma \ref{lem:covar} provides
$$ {\rm Var}\,\wh{N}(t)~\sim~ \mu^{-1-1/\alpha}c_{\alpha}(t)\int_{\R} \big(\Prob\{\mathcal{S}_{\alpha}(1)\le y\}-(\Prob\{\mathcal{S}_{\alpha}(1)\le y\})^{2}\big)\ {\rm d}y\quad\text{as }t\to\infty,$$
and the value of the integral is calculated in Lemma \ref{lem:calcintegr}.
\end{proof}

\subsection{Proof of Theorem \ref{thm:main}}

For $t>0$ sufficiently large, we consider the process
$$ Z(t,u)\,:\,=\frac{\wh{N}(h_{\alpha}(t+u))-V(h_{\alpha}(t+u)) 
}{(\mu^{-1-1/\alpha} c_{\alpha}(h_{\alpha}(t)))^{1/2}},\quad u\in\R. $$
By the Cram\'{e}r-Wold device, the weak convergence of its finite-dimensional distributions is equivalent to
\begin{equation}\label{eq:Cramer-Wold device}
\sum_{i=1}^{k}\lambda_{j} Z(t,u_{i})~{\overset{{\rm d}}\longrightarrow}~ \sum_{i=1}^{k} \lambda_{i} X_{\alpha}(u_{i})\quad\text{as }t\to\infty
\end{equation}
for all $k\in\N$, all $\lambda_{1},\ldots, \lambda_{k}\in\R$ and all $-\infty<u_{1}<\ldots<u_{k}<\infty$. The left-hand side in \eqref{eq:Cramer-Wold device} is equal to
$$ \frac{\sum_{n\ge 1}\sum_{i=1}^{k} \lambda_{i}(\1_{\{\wh{S}_{n}\le h_{\alpha}(t+u_{i})\}}-\Prob\{\wh{S}_{n}\le h_{\alpha}(t+u_{i})\})}{(\mu^{-1-1/\alpha}c_{\alpha}(h_{\alpha}(t)))^{1/2}}$$ and as such an infinite sum of independent centered random variables with finite second moments. Hence, in order to prove \eqref{eq:Cramer-Wold device}, it suffices to show (see, for instance, Thm.\,3.4.5 on p.~129 in 
\cite{Durrett:2010}) that
\begin{align}
\begin{split}\label{eq:mgale CLT1}
&\hspace{1cm}\lim_{t\to\infty}\Erw\bigg(\sum_{i=1}^{k} \lambda_{i} Z(t,u_{i})\bigg)^{2}\ =\ \Erw\bigg(\sum_{i=1}^{k} \lambda_{j} X_{\alpha}(u_{i})\bigg)^{2}\\
&=\ \sum_{i=1}^{k}\lambda_{i}^{2}\, {\rm Var}\,X_{\alpha}(u_{i})\,+\,2\sum_{1\le i<j\le k}\lambda_{i} \lambda_{j}\,{\rm Cov}\,(X_{\alpha}(u_{i}), X_{\alpha}(u_{j}))
\end{split}
\end{align}
and
\begin{align}\label{eq:mgale CLT2}
\lim_{t\to\infty}\sum_{n\ge 1}\Erw\Bigg(\frac{\big[\sum_{i=1}^{k} \lambda_{i}(\1_{\{\wh{S}_{n}\le h_{\alpha}(t+u_{i})\}}-\Prob\{\wh{S}_{n}\le h_{\alpha}(t+u_{i})\})\big]^{2}}{c_{\alpha}(h_{\alpha}(t))}\1_{E_{n}(t)}\Bigg)\ =\ 0
\end{align}
for all $\eps>0$, where
$$ E_{n}(t)\ :=\ \Bigg\{\Bigg|\sum_{i=1}^{k}\lambda_{i}(\1_{\{\wh{S}_{n}\le h_{\alpha}(t+u_{i})\}}-\,\Prob\{\wh{S}_{n}\le h_{\alpha}(t+u_{i})\})\Bigg|>\eps (c_{\alpha}(h_{\alpha}(t)))^{1/2}\Bigg\}. $$
Eq.\,\eqref{eq:mgale CLT1} follows immediately from Lemma \ref{lem:covar}, and  \eqref{eq:mgale CLT2} is a consequence of
\begin{align}\label{eq:mgale CLT2 aux}
\lim_{t\to\infty}\sum_{n\ge 1}\Erw\Bigg(\frac{(\1_{\{\wh{S}_{n}\le h_{\alpha}(t+u)\}}-\Prob\{\wh{S}_{n}\le h_{\alpha}(t+u)\})^{2}}{c_{\alpha}(h_{\alpha}(t))}\1_{E_{n}'(t)}\Bigg)\ =\ 0,
\end{align}
for fixed $u\in\R$ and $E_{n}'(t):=\{|\1_{\{\wh{S}_{n}\le h_{\alpha}(t+u)\}}-\,\Prob\{\wh{S}_{n}\le h_{\alpha}(t+u)\})|>\eps (c_{\alpha}(h_{\alpha}(t)))^{1/2}\}$, when using the inequality
\begin{align}
(a_{1}+\ldots+a_{k})^{2}&\1_{\{|a_{1}+\ldots+a_{k}|>y\}}\ \le\
(|a_{1}|+\ldots+|a_{k}|)^{2}\1_{\{|a_{1}|+\ldots+|a_{k}|>y\}}\notag\\
&\le\ k^{2} (|a_{1}|\vee\ldots\vee |a_{k}|)^{2}\1_{\{k(|a_{1}|\vee\ldots\vee |a_{k}|)>y\}}\notag\\
&\le\ k^{2}\big(a_{1}^{2}\1_{\{|a_{1}|>y/k\}}+\ldots+a_{k}^{2}\1_{\{|a_{k}|>y/k\}}\big),\label{eq:tech}
\end{align}
valid for all real $a_{1},\ldots, a_{k}$ and $y>0$. As for \eqref{eq:mgale CLT2 aux}, we note that it trivially holds because $|\1_{\{\wh{S}_{n}\le h_{\alpha}(t+u)\}}-\Prob\{\wh{S}_{n}\le h_{\alpha}(t+u)\}|\le 1$ a.s.~and therefore the indicator $\1_{E_{n}'(t)}$ equals 0 for sufficiently large $t$. This completes the proof of \eqref{eq:Cramer-Wold device}.

\vspace{.2cm}
Assume now that $V$ is Lipschitz continuous on $[0,\infty)$, thus
\begin{equation}\label{eq:Lip}
|V(t)-V(s)|\,\le\, C|t-s|\quad\text{for all }t,s\ge 0\text{ and some }C>0.
\end{equation}
The subsequent proof is similar to that of Theorem 1.1 in \cite{Iksanov+Kabluchko+Kotelnikova:2022}, where an infinite sum of other independent indicators was investigated. We intend to prove that the family of distributions of the processes $(Z(t,u))_{u\in\R}$, $t>0$, is tight in the Skorokhod space $D[-A,\,A]$ for any fixed $A>0$. To this end, we will show that there is a constant $C_{1}>0$ such that
\begin{equation}\label{eq:tight}
\Erw (Z(t,v)-Z(t,u))^{2} (Z(t,w)-Z(t,v))^{2}\ \le\ C_{1} (w-u)^{2}
\end{equation}
for all $u<v<w$ in the interval $[-A,\,A]$ and sufficiently large $t>0$. Together with the already shown fact that $Z(t,0)$ converges in law as $t\to\infty$, this implies the claimed tightness by a well-known sufficient condition (see Theorem 13.5 and formula (13.14) on p.~143 in~\cite{Billingsley:1999}).

\vspace{.1cm}
For $n\in\N$, we introduce the Bernoulli random variables
\begin{equation}\label{eq:B_{k}_C_{k}_def}
L_{n}\,:=\,\1_{\{h_{\alpha}(t+u)<\wh{S}_{n} \le h_{\alpha}(t+v)\}}\quad\text{and}\quad M_{n}\,:=\,\1_{\{h_{\alpha}(t+v)< \wh{S}_{n} \le h_{\alpha}(t+w)\}}
\end{equation}
along with their centered versions
$$ \ovl{L}_{n}\,:=\,L_{n}-\Erw L_{n}\quad\text{and}\quad\ovl{M}_{n}\,:=\,M_{n}- \Erw M_{n}. $$
Notice that the dependence of these variables on $u,v,w$ and $t$ is not shown. Let also
$$ q_{n}\,:=\,\Prob\{L_{n} = 1\}\,=\,\Erw L_{n}\quad\text{and}\quad z_{n}\,:=\,\Prob\{M_{n} = 1\}\,=\,\Erw M_{n}. $$
Owing to \eqref{eq:Lip},
\begin{gather}
\sum_{n\ge 1}q_{n}\,=\,V(h_{\alpha}(t+v))-V(h_{\alpha}(t+u))\,\le\,C(v-u)\sup_{z\in [t-A,\,t+A]}\,h_{\alpha}'(z)\label{eq:q_{k}_est}
\shortintertext{and}
\sum_{n\ge 1}z_{n}\,=\,V(h_{\alpha}(t+w))-V(h_{\alpha}(t+v))\,\le\,C(w-v)\sup_{z\in [t-A,\,t+A]}\,h_{\alpha}'(z).\label{eq:z_{k}_est}
\end{gather}
Recalling $b_{\alpha}(t)=\mu^{-1-1/\alpha} c_{\alpha}(h_{\alpha}(t))$, we observe that
\begin{gather}
b_{\alpha}(t)^{1/2}(Z(t,v)-Z(t,u))\ =\ \sum_{n\ge 1} \ovl{L}_{n}
\shortintertext{and}
b_{\alpha}(t)^{1/2}((Z(t,w)-Z(t,v))\ =\ \sum_{n\ge 1} \ovl{M}_{n},
\end{gather}
which implies that \eqref{eq:tight} is equivalent to
$$ \Erw\Bigg(\sum_{n_{1}\ge 1}\ovl{L}_{n_{1}}\Bigg)^{2} \Bigg(\sum_{n_{2}\ge 1} \ovl{M}_{n_{2}}\Bigg)^{2} \le C_{1} (w-u)^{2}b_{\alpha}(t)^{2} $$
for all $u<v<w$ in the interval $[-A,\,A]$ and large $t>0$. After term-wise multiplication, our task reduces to showing that
$$ \sum_{n_{1},n_{2},n_{3},n_{4}\ge 1}\Erw\Big[\ovl{L}_{n_{1}}\ovl{L}_{n_{3}}\ovl{M}_{n_{2}}\ovl{M}_{n_{4}}\Big]\ \le\ C_{1} (w-u)^{2}\,b_{\alpha}(t)^{2}. $$
If one number of $n_{1},\ldots,n_{4}$ appears only once in $(n_{1},n_{2},n_{3},n_{4})$, then $\Erw\big[\ovl{L}_{n_{1}}\ovl{L}_{n_{3}}\ovl{M}_{n_{2}}\ovl{M}_{n_{4}}\big]=0$ because the variable with that index number is independent of the other variables in the random vector $(\ovl{L}_{n_{1}}\ovl{L}_{n_{3}}\ovl{M}_{n_{2}}\ovl{M}_{n_{4}})$. This leaves us with the consideration of those $(n_{1},n_{2},n_{3},n_{4})$ in which every number appears at least twice.

\vspace*{2mm}
\textsc{Case 1.}
We begin with the case when $n_{1}\neq n_{3}$. Then, either $n_{2} =n_{1}$ and $n_{4} =n_{3}$ must hold, or $n_{2} =n_{3}$ and $n_{4} =n_{1}$. We only investigate the first situation, the second being similar. The corresponding contribution is
$$ \sum_{n_{1}\neq n_{3}} \Erw\Big[\ovl{L}_{n_{1}}\ovl{L}_{n_{3}}\ovl{M}_{n_{1}}\ovl{M}_{n_{3}}\Big]=\sum_{n_{1}\neq n_{3}} \Erw\Big[\ovl{L}_{n_{1}}\ovl{M}_{n_{1}}\Big] \Erw\Big[ \ovl{L}_{n_{3}}\ovl{M}_{n_{3}}\Big].$$ Since $L_{n_{1}}$ and $M_{n_{1}}$ cannot be equal to $1$ at the same time, we infer $L_{n_{1}}M_{n_{1}}=0$ and thereupon
$$ \Erw\Big[\ovl{L}_{n_{1}}\ovl{M}_{n_{1}}\Big]\ =\ -\,\Erw L_{n_{1}}\Erw M_{n_{1}}=-q_{n_{1}}z_{n_{1}}. $$
By the same argument, $\Erw\big[\ovl{L}_{n_{3}}\ovl{M}_{n_{3}}\big]=-q_{n_{3}}z_{n_{3}}$. It follows that
$$ \sum_{n_{1}\neq n_{3}} \Erw\Big[\ovl{L}_{n_{1}}\ovl{L}_{n_{3}}\ovl{M}_{n_{1}}\ovl{M}_{n_{3}}\Big]\ =\ \sum_{n_{1}\neq n_{3}} q_{n_{1}}z_{n_{1}}q_{n_{3}}z_{n_{3}}\ \le\ \sum_{n_{1}\ge 1}q_{n_{1}}\sum_{n_{2}\ge 1}z_{n_{2}}. $$
By invoking \eqref{eq:q_{k}_est} and \eqref{eq:z_{k}_est}, we arrive at
\begin{equation}\label{eq:est_sum_q_r}
\sum_{n_{1}\ge 1} q_{n_{1}}\sum_{n_{2}\ge 1}z_{n_{2}}\ \le\ C^{2} (w-u)^{2} \bigg(\sup_{z\in [t-A,\,t+A]}h_{\alpha}'(z)\bigg)^{2}\ \le\ C_{1}(w-u)^{2}b_{\alpha}(t)^{2}
\end{equation}
for all $u<v<w$ in the interval $[-A,\,A]$, all sufficiently large $t>0$ and a suitable $C_{1}>0$. Here we have used
$$ \lim_{t\to\infty}\frac{\sup_{z\in [t-A,\,t+A]}\,h_{\alpha}'(z)}{c_{\alpha}(h_{\alpha}(t))}\ =\ \frac{\alpha}{\alpha-1} $$
which is guaranteed by \eqref{eq:diff}.

\vspace*{2mm}
\textsc{Case 2.} Let now $n_{1} = n_{3}$, and also $n_{2}=n_{4}$, for otherwise $\Erw \big[\ovl{L}_{n_{1}}\ovl{L}_{n_{3}}\ovl{M}_{n_{2}}\ovl{M}_{n_{4}}\big]=0$. Then
\begin{align*}
\sum_{n_{1},n_{2}\ge 1}\Erw\Big[\ovl{L}_{n_{1}} \ovl{L}_{n_{1}} \ovl{M}_{n_{2}}\ovl{M}_{n_{2}}\Big]\ &=\ \sum_{n_{1}\neq n_{2}} \Erw\big[\ovl{L}_{n_{1}}^{2} \big]\Erw\big[\ovl{M}_{n_{2}}^{2}\big]+\sum_{n\ge 1}\Erw\Big[\ovl{L}_{n}^{2}\ovl{M}_{n}^{2}\Big]
\\
&\le\ \sum_{n_{1}\neq n_{2}} q_{n_{1}}z_{n_{2}}+2\sum_{n\ge 1} q_{n} z_{n}\ \le\ 2 \sum_{n_{1}\ge 1 }q_{n_{1}}\sum_{n_{2}\ge 1}z_{n_{2}}\\
&\le\ C_{1}(w-u)^{2}b_{\alpha}(t)^{2}  
\end{align*}
for all $u<v<w$ in the interval $[-A,\,A]$, all sufficiently large $t>0$ and some $C_{1}>0$. The first equality holds because $L_{n}$ and $M_{n}$ cannot be equal to $1$ simultaneously, and the last inequality is just \eqref{eq:est_sum_q_r}. Regarding the first inequality, one has to combine
\begin{gather*}
\Erw [\ovl{L}_{n_{1}}^{2}]\ =\ q_{n_{1}}(1-q_{n_{1}})\le q_{n_{1}},\qquad\Erw [\ovl{M}_{n_{2}}^{2}]\ =\ z_{n_{2}}(1-z_{n_{2}})\le z_{n_{2}}
\shortintertext{and}
\begin{split}
\Erw\Big[\ovl{L}_{n}^{2}\ovl{M}_{n}^{2}\Big]\ &=\ q_{n} (1-q_{n})^{2}(-z_{n})^{2} +z_{n}(1-z_{n})^{2} (-q_{n})^{2} + (1-q_{n}-z_{n})(-q_{n})^{2}(-z_{n})^{2}\\
&=q_{n}z_{n}(q_{n}+z_{n}-3q_{n}z_{n})\le 2 q_{n}z_{n}.
\end{split}
\end{gather*}
This completes the proof of tightness, and the convergence
$$ \Big(\frac{\wh{N}(h_{\alpha}(t+u))-V(h_{\alpha}(t+u))}{(\mu^{-1-1/\alpha} c_{\alpha}(h_{\alpha}(t)))^{1/2}}\Big)_{u\in\R}~\Longrightarrow ~(X_{\alpha}(u))_{u\in\R}\quad\text{as }t\to\infty $$
in the $J_{1}$-topology on $D(\R)$ follows as a consequence, which in turn completes the proof of Theorem \ref{thm:main}.

\subsection{Proof of Theorem \ref{thm:integral}}

By the properties of stochastic integrals with respect to white noise, $Y_{\alpha}$ is a stationary centered Gaussian process, and its covariance function equals
\begin{align*}
{\rm Cov}\,&(Y_{\alpha}(u), Y_{\alpha}(v))\\
&=\ \int_{\R}\int_{[0,\,1]} (\1_{\{y \le \Phi_{\alpha}(a_{\alpha} u+x)\}} - \Phi_{\alpha}(a_{\alpha} u+x))(\1_{\{y \le \Phi_{\alpha}(a_{\alpha} v+x)\}} - \Phi_{\alpha}(a_{\alpha} v+x))\ {\rm d}x\ {\rm d}y\\
&=\ \int_{\R} (\Phi_{\alpha}(a_{\alpha} (u\wedge v)+x)-\Phi_{\alpha}(a_{\alpha} (u\wedge v)+x)\Phi_{\alpha}(a_{\alpha} (u\vee v)+x))\ {\rm d}x
\end{align*}
for $u,v\in\R$, where the last line follows with the help of the formula
$$ \int_{0}^{1}(\1_{\{y\le a\}}-a)(\1_{\{y<b\}}-b)\ {\rm d}y\ =\ a\wedge b - ab, $$
valid for all $a,b\in [0,1]$. By finally integrating with respect to $x$, we obtain
\begin{align*}
&{\rm Cov}\,(Y_{\alpha}(u),Y_{\alpha}(v))\\
&= \int_{\R}(\Phi_{\alpha}(a_{\alpha} (u\wedge v)+x)-\Phi_{\alpha}(a_{\alpha} (u\wedge v)+x)\Phi_{\alpha}(a_{\alpha} (u\vee v)+x))\ {\rm d}x\\
&=\ \int_{\R}\Prob\{\mathcal{S}_{\alpha}(1)\le a_{\alpha} (u\wedge v)+x\}\Prob\{\mathcal{S}_{\alpha}(1)>a_{\alpha} (u\vee v)+x)\}\ {\rm d}x\ =\ {\rm Cov}\,(X_{\alpha}(u),X_{\alpha}(v)).
\end{align*}

By showing next that $\Erw (Y_{\alpha}(u)-Y_{\alpha}(0))^{2}\sim a_{\alpha} |u|$ as $u\to 0$, the claim about H\"{o}lder continuity of the paths follows from the Kolmogorov-Chentsov theorem. Assume that $u>0$. We recall from Theorem 1 in \cite{Yamazato:1978} that the law of $\mathcal{S}_{\alpha}(1)$ is unimodal and has a continuous density $g_{\alpha}$, say. With this at hand, we infer with the help of the monotone convergence theorem, for some $\theta(y,u)\in [y, a_{\alpha} u+y]$,
\begin{align*}
\Erw (Y_{\alpha}(u)-Y_{\alpha}(0))^{2}\ &=\ 2\int_{\R}(\Prob\{\mathcal{S}_{\alpha}(1)>y\}-\Prob\{\mathcal{S}_{\alpha}(1)>a_{\alpha} u+y\})\Prob\{\mathcal{S}_{\alpha}(1)\le y\}\ {\rm d}y\\
&=\ 2a_{\alpha} u\int_{\R}g_{\alpha}(\theta (y,u))\Prob\{\mathcal{S}_{\alpha}(1)\le y\}\ {\rm d}y\quad\text{for some }\theta(y,u)\in [y, a_{\alpha} u+y]\\
&\xrightarrow{u \searrow 0}~2a_{\alpha} u\int_{\R} g_{\alpha}(y)\Prob\{\mathcal{S}_{\alpha}(1)\le y\}\ {\rm d}y\ =\ a_{\alpha} u.
\end{align*}
For aesthetic reasons we write $u\searrow 0$ in place of $u\to 0+$. The case $u<0$ can be treated similarly.

\section{Proofs for Section \ref{sect:minimum}}

In order to prove Theorem \ref{thm:min}, we need the following lemma.
\begin{lemma}\label{lem:aux}
Assume that the law of $\xi$ belongs to the domain of attraction of an $\alpha$-stable law for some $\alpha\in (1,2]$. Then $\mu=\Erw\xi<\infty$ and
$$ -\log \prod_{n\ge \lfloor t/\mu\rfloor+1}\Prob\{S_{n}>t\}\ =\ O(t)\quad\text{as }t\to\infty $$
\end{lemma}

\begin{proof}
The proof mimics the one given on pp.~123-124 in \cite{Hough+Krishnapur+Peres+Virag:2009} for the particular case where $\xi$ has an exponential law. In order to simplify the subsequent presentation, we omit the use of integer parts and write, for example, $t/\mu$ instead of $\lfloor t/\mu\rfloor$. This does not affect the asymptotics.

\vspace{.1cm}
By assumption, $\lim_{t\to\infty}\Prob\{S_{t/\mu}>t\}=\Prob\{X_{\alpha}>0\}=:c_{\alpha}>0$, where $X_{\alpha}$ is an an $\alpha$-stable random variable. As a consequence, $\Prob\{S_{t/\mu}>t\}\ge c_{\alpha}/2$ for all sufficiently large $t$ and therefore
$$ \prod_{n=t/\mu+1}^{2t/\mu}\Prob\{S_{n}>t\}\ \ge\ \prod_{n=t/\mu+1}^{2t/\mu}\Prob\{S_{t/\mu}>t\}\ \ge\ (c_{\alpha}/2)^{t/\mu}. $$
This proves $-\log \prod_{n=t/\mu+1}^{2t/\mu}\Prob\{S_{n}>t\}=O(t)$ as $t\to\infty$.

\vspace{.1cm}
For large $t$ and all $n\ge 2t/\mu$, we further have $\Prob\{S_{n}\le t\}\le 1-c_{\alpha}/2$. As $-\log(1-x)\le (2/c_{\alpha})x$ for all $x\in [0,1-c_{\alpha}/2]$, it follows
$$ -\log\prod_{n\ge 2t/\mu}\Prob\{S_{n}>t\}\le (2/c_{\alpha})\sum_{n\ge 2t/\mu}\Prob\{S_{n}\le t\}\ \le\ (2/c_{\alpha}) \sum_{n\ge 2t/\mu}\Prob\{S_{n}\le \mu n/2\}, $$
and since $\lim_{u\to 0+}u^{-1}(-\log \Erw\exp(-u\xi))=\mu$, we conclude that $c:=\mu/2+u_{0}^{-1}\log \Erw\exp(-u_{0} \xi)\in (-\infty,0)$ for some $u_{0}>0$. With the help of Markov's inequality, this yields
\begin{align*}
\sum_{n\ge 2t/\mu}\Prob\{S_{n}\le \mu n/2\}\ &\le\ \sum_{n\ge 2t/\mu}\exp(u_{0}\mu n/2)\,\Erw\exp(-u_{0} S_{n})\\
&=\ \sum_{n\ge 2t/\mu}\exp(cn)\ =\ o(1)\quad\text{as }t\to\infty
\end{align*}
and thus $-\log\prod_{n\ge 2t/\mu}\Prob\{S_{n}>t\}=o(1)$ as $t\to\infty$, which completes the proof.
\end{proof}

\begin{proof}[Proof of Theorem \ref{thm:min}]
The assumptions of the theorem obviously ensure that $\xi$ is square-inte\-grable, which in turn implies that its law belongs to the normal domain of attraction of a normal distribution. By Lemma \ref{lem:aux}, it is therefore enough to examine the logarithmic asymptotics of the product $\prod_{n=1}^{\lfloor t/\mu\rfloor}\Prob\{S_{n}>t\}$.

\vspace{.2cm}
(a) For any $a\in (0, 1/\mu)$ and $t>1/a$,
\begin{align*}
-t^{2}\log\prod_{n=1}^{t/\mu}\Prob\{S_{n}>t\}\ &=\ t^{2}\int_{1}^{t/\mu}-\log \Prob\{S_{\lfloor y\rfloor}>t\}\ {\rm d}y\\
&=\ \int_{1/t}^{1/\mu}\frac{-x\log \Prob\{S_{\lfloor tx\rfloor}>t\}}{tx}\ {\rm d}x\\
&=\ \bigg[\int_{1/t}^{a}+\int_{a}^{1/\mu}\bigg]\frac{-x\log \Prob\{S_{\lfloor tx\rfloor}>t\}}{tx}\ {\rm d}x.
\end{align*}
By Cram\'er's theorem (see, for instance, Thm.\,I.4 in 
\cite{Hollander:2000}),
$$ \lim_{t\to\infty}\frac{-\log \Prob\{S_{\lfloor tx\rfloor}>t\}}{tx}\ =\ I(1/x)\quad\text{for any }x\in [a,1/\mu], $$
and the convergence is uniform in $x\in [a, 1/\mu]$ because $x\mapsto I(1/x)$ is continuous  ($I$ is convex) and $x\mapsto -\log \Prob\{S_{\lfloor tx\rfloor}>t\}/(tx)$ is nonincreasing for each $t$. Consequently,
$$ \lim_{t\to\infty} \int_{a}^{1/\mu}\frac{-x\log \Prob\{S_{\lfloor tx\rfloor}>t\}}{tx}\ {\rm d}x\ =\ \int_{a}^{1/\mu}xI(1/x){\rm d}x. $$
for any $a\in (0,1/\mu)$. To complete the proof of (a), it is therefore enough to prove
$$ \lim_{a\searrow 0}\limsup_{t\to\infty}\int_{1/t}^{a}\frac{-\log \Prob\{S_{\lfloor tx\rfloor}>t\}}{t}\ {\rm d}x\ =\ 0. $$
By using
\begin{equation}\label{eq:ineq}
\Prob\{S_{n}>t\}\ \ge\ \Prob\Big\{\min_{1\le k\le n}\,\xi_{k}>t/n\Big\}\ =\ \big(\Prob\{\xi>t/n\}\big)^n,\quad t\ge 0,~ n\in\N
\end{equation}
and neglecting for simplicity integer parts, we infer
$$ \int_{1/t}^{a}\frac{-\log \Prob\{S_{\lfloor tx\rfloor}>t\}}{t}\ {\rm d}x\ \le\ \int_{0}^{a}-x\log\Prob\{\xi>1/x\}\ {\rm d}x, $$
and thus arrive at the desired conclusion because \eqref{eq:integral} ensures that the right-hand integral vanishes as $a\searrow 0$.

\vspace{.2cm}
(b) Assuming \eqref{eq:log tail condition}, Kasahara's Tauberian theorem (Thm.\,4.12.7 in
\cite{Bingham+Goldie+Teugels:1989}) provides, for all $n\in\N$ and some $c>0$,
\begin{gather}
\lim_{t\to\infty}\frac{-\log \Prob\{S_{n}>t\}}{t^\alpha \ell(t)}\ =\ cn^{1-\alpha},\nonumber
\shortintertext{hence}
\lim_{t\to\infty}\frac{-\log \prod_{k=1}^{n}\Prob\{S_{k}>t\}}{t^\alpha \ell(t)}\ =\ c\sum_{k=1}^{n} k^{1-\alpha}.\label{eq:early}
\end{gather}

(b1) We first note that $\lim_{t\to\infty}\ell^{*}(t)=\infty$ is equivalent to $\int_{0}^{1} -y\log \Prob\{\xi>1/y\}{\rm d}y=\infty$. Further, by Prop.\,1.5.9(a) in \cite{Bingham+Goldie+Teugels:1989}, $\lim_{t\to\infty}(\ell^{*}(t)/\ell(t))=\infty$. Now, for $t>0$ and any fixed $n\in\N$,
\begin{align*}
-\log \prod_{k=n}^{\lfloor t/\mu\rfloor}\Prob\{S_{k}>t\}\ &=\ \int_{n}^{\lfloor t/\mu\rfloor}-\log \Prob\{S_{\lfloor y\rfloor}>t\}\ {\rm d}y\\
&=\ t\int_{n/t}^{\lfloor t/\mu\rfloor/t}-\log \Prob\{S_{\lfloor tx\rfloor}>t\}\ {\rm d}x\\
&\le\ t^{2} \int_{n/t}^{\lfloor t/\mu\rfloor/t}\big(\lfloor tx\rfloor/t\big)(-\log \Prob\{\xi>t/\lfloor tx\rfloor\})\ {\rm d}x\\
&\le\ t^{2}\int_{n/t}^{1/\mu} -x\log \Prob\{\xi>t/\lfloor tx\rfloor\}\ {\rm d}x  
\end{align*}
having utilized \eqref{eq:ineq} for the penultimate inequality. Given $\eps\in (0,1)$, there exists $n_{0}\in\N$ such that $(1-\eps)\lfloor x\rfloor\le x\le (1+\eps)\lfloor x\rfloor$ whenever $x\ge n_{0}$ (the left-hand inequality will be used later). Since $x\mapsto -\log \Prob\{\xi>x\}$ is nondecreasing, the right-hand inequality provides
$$ -\log \Prob\{\xi>(tx/\lfloor tx \rfloor)(1/x)\}\ \le\ -\log\Prob\{\xi>(1+\eps)/x\} $$
for all $x\ge n_{0}/t$. By combining these facts, we obtain
\begin{equation*}
-\log\prod_{k=n_{0}}^{\lfloor t/\mu\rfloor}\Prob\{S_{k}>t\}\ \le\ t^{2}\int_{n_{0}/t}^{1/\mu}-x\log \Prob\{\xi>(1+\eps)/x\}\ {\rm d}x~\sim~(1+\eps)^{2} ct^{2}\ell^{*}(t)\quad\text{as }t\to\infty.
\end{equation*}
If $\alpha=2$, Eq.\,\eqref{eq:early} tells us that the contribution of $\prod_{k=1}^{n_{0}-1}\Prob\{S_{k}>t\}$ is negligible in comparison to that of $\prod_{k=n_{0}}^{\lfloor t/\mu\rfloor}\Prob\{S_{k}>t\}$ as $t\to\infty$. Hence,
$$ \limsup_{t\to\infty}\frac{-\log\prod_{k=1}^{\lfloor t/\mu\rfloor}\Prob\{S_{k}>t\}}{t^{2}\ell^{*}(t)}\ \le\ (1+\eps)^{2}c $$
for all $\eps\in (0,1)$, that is
$$ \limsup_{t\to\infty}\frac{-\log \prod_{n=1}^{\lfloor t/\mu\rfloor}\Prob\{S_{n}>t\}}{t^{2}\ell^{*}(t)}\ \le\ c.$$

It remains to prove the reverse inequality for the lower limit. By Markov's inequality
$$ -\log \Prob\{S_{n}>t\}\ \ge\ n(u(t/n)-\log \Erw\exp(u\xi)). $$
for $n\in\N$ and $t,u>0$, which proves
\begin{equation}\label{eq:ineq2}
-\log \Prob\{S_{n}>t\}\ \ge\ n I(t/n),\quad n\in\N,~t>0.
\end{equation}
Using this, we obtain, for $n_{0}\in\N$ as above and $t\le \mu n_{0}$,
\begin{align*}
-\log \prod_{k=n_{0}}^{\lfloor t/\mu\rfloor}\Prob\{S_{k}>t\}\ &=\ t\int_{n_{0}/t}^{\lfloor t/\mu\rfloor/t}(-\log \Prob\{S_{\lfloor tx\rfloor}>t\})\ {\rm d}x\\
&\ge\ t^{2} \int_{n_{0}/t}^{\lfloor t/\mu\rfloor/t} (\lfloor tx\rfloor/t)I(t/\lfloor tx\rfloor){\rm d}x\\
&\ge (1-\eps) t^{2}\int_{n_{0}/t}^{\mu^{-1}(1-\eps)}xI((1-\eps)/x)\ {\rm d}x,
\end{align*}
where the last inequality holds because $x\mapsto I(1/x)$ is nonincreasing on $(0, \mu^{-1}(1-\eps))$.

\vspace{.1cm}
The following lemma provides the asymptotic behavior of the rate function $I(x)$.

\begin{lemma}\label{lem:two}
Assume
$$ -\log \Prob\{\xi>t\}\,\sim\,c t^{2}\ell(t)\quad\text{as }t\to\infty. $$
for some $c>0$ and a slowly varying function $\ell$ satisfying \eqref{eq:bojanic}. Then
$$ \log I(t)\,\sim\,-\log \Prob\{\xi>t\}\,\sim\,c t^{2}\ell(t)\quad\text{as }t\to\infty. $$
\end{lemma}
\begin{proof}
By Theorem 2.3.3 in \cite{Bingham+Goldie+Teugels:1989}, relation \eqref{eq:bojanic} entails $\lim_{t\to\infty}(\ell(t(\ell(t))^\beta)/\ell(t))=1$ for all $\beta\in\R$. Using this with $\beta=2$ and $\beta=1/2$, we infer with the help of Cor.\,2.3.4 in \cite{Bingham+Goldie+Teugels:1989} and Kasahara's Tauberian theorem (Thm.\,4.12.7 in \cite{Bingham+Goldie+Teugels:1989})
$$ \log \Erw\exp(s\xi)~\sim~ \frac{cs^{2}}{4\ell(s)}~\sim~\frac{cs^{2}\ell^{\#}(s)}{4}\quad\text{as }s\to\infty, $$
where $\ell^{\#}$ denotes the de Bruijn conjugate of $\ell$, see p.\,29 in \cite{Bingham+Goldie+Teugels:1989} for the definition. We now arrive at the claim by an appeal to Thm.\,1.8.10 in \cite{Bingham+Goldie+Teugels:1989}.
\end{proof}

\vspace{.1cm}
A combination of the previous lemma with \eqref{eq:early} provides
$$ \liminf_{t\to\infty}\frac{-\log \prod_{n=1}^{\lfloor t/\mu\rfloor}\Prob\{S_{n}>t\}}{t^{2}\ell^{*}(t)}\ \ge\ (1-\eps)^3c $$
for all $\eps\in (0,1)$ and then
$$ \liminf_{t\to\infty}\frac{-\log \prod_{n=1}^{\lfloor t/\mu\rfloor}\Prob\{S_{n}>t\}}{t^{2}\ell^{*}(t)}\ \ge\ c. $$

(b2) In view of \eqref{eq:early}, it suffices to show that
\begin{equation}\label{eq:inter1}
\lim_{n\to\infty}\limsup_{t\to\infty}\frac{-\log \prod_{k=n}^{\lfloor t/\mu\rfloor} \Prob\{S_{k}>t\}}{t^\alpha \ell(t)}=0.
\end{equation}
Using again \eqref{eq:ineq} while ignoring integer parts, we conclude that 
\begin{align*}
-\log \prod_{k=n}^{t/\mu}\Prob\{S_{k}>t\}\ &=\ t\int_{n/t}^{1/\mu}-\log \Prob\{S_{\lfloor ty\rfloor}>t\}\ {\rm d}y\\
&\le\ t^{2}\int_{n/t}^{1/\mu}-y\log\Prob\{\xi>1/y\}{\rm d}y\\
&=\ t^{2}\int_{\mu}^{t/n}-y^{-3}\log \Prob\{\xi>y\}{\rm d}y\\
&\sim\ (\alpha-2)^{-1}(t/n)^\alpha \ell(t)\quad\text{as }t\to\infty,
\end{align*}
and this obviously shows \eqref{eq:inter1}. The asymptotic relation is ensured by Karamata's theorem (Prop.\,1.5.8 in \cite{Bingham+Goldie+Teugels:1989}).
\end{proof}

\begin{proof}[Proof of Theorem \ref{thm:heavy}]
(a) Put $c(t):=\Prob\{\xi>t\}$ and recall that $\tau(t)=\inf\{k\in\N: S_{k}>t\}$ for $t\ge 0$. For any positive $a$ and $b$, $a<b$, and all sufficiently large $t$,
\begin{align*}
\int_{1}^{\infty}-\log \Prob\{S_{\lfloor x\rfloor}>t\}\ {\rm d}x\ &=\ (c(t))^{-1}\int_{c(t)}^{\infty}-\log\Prob\{S_{\lfloor x/c(t)\rfloor}>t\}\ {\rm d}x\\
&=\ (c(t))^{-1}\int_{c(t)}^{\infty}-\log\Prob\{\tau(t)\le {\lfloor x/c(t)\rfloor}\}\ {\rm d}x\\
&=\ (c(t))^{-1}\bigg(\int_{c(t)}^{a}+\int_{a}^{b}+\int_{b}^{\infty}\bigg)\ldots
\end{align*}
Under the assumption of part (a), $c(t)\tau(t)$ converges in distribution to $W_{\alpha}^{\leftarrow}(1)$ as $t\to\infty$ (see, for instance, Thm.\,7 in \cite{Feller:1949}), and the convergence
\begin{equation}\label{eq:distrconv}
\lim_{t\to\infty} \log \Prob\{\tau(t)\le {\lfloor x/c(t)\rfloor}\}\ =\ \log \Prob\{W_{\alpha}^{\leftarrow}(1)\le x\}
\end{equation}
is uniform in $x\in [a,b]$ by Poly\`a's theorem (the law of the limit is continuous). This entails
$$ \int_{a}^{b}-\log\Prob\{\tau(t)\le {\lfloor x/c(t)\rfloor}\}\ {\rm d}x\ =\ \int_{a}^{b}-\log\Prob\{W_{\alpha}^{\leftarrow}(1)\le x\}\ {\rm d}x. $$

Fixing some $p\in (0,1)$ such that $-\log(1-x)\le 2x$ for all $x\in [0, p]$, a simple tightness~argument provides $\sup_{x\ge b}\Prob\{c(t)\tau(t)>x\}\le p$ when choosing $b$ and $t$ sufficiently large. It follows
$$ \int_b^{\infty}-\log \Prob\{ c(t)\tau(t)\le x\}\ {\rm d}x\ \le\ 2\int_b^{\infty}\Prob\{c(t)\tau(t)>x\}\ {\rm d}x\ \le\ 2\sup_{t\ge 1}\Erw (c(t)\tau(t))^{2} \int_{b}^{\infty} x^{-2}\ {\rm d}x, $$
where integer parts have again been omitted. As $\sup_{t\ge 1} \Erw (c(t)\tau(t))^{2}<\infty$ by Thm.\,1.5 in \cite{Iksanov+Marynych+Meiners:2016}, we see that
$$ \lim_{b\to\infty}\limsup_{t\to\infty}\int_b^{\infty}-\log \Prob\{\tau(t)\le {\lfloor x/c(t)\rfloor}\}\ {\rm d}x\ =\ 0. $$

Turning to $\int_{c(t)}^{a}-\log \Prob\{\tau(t)\le\lfloor x/c(t)\rfloor\}\,{\rm d}x$, note first that $1-\eee^{-y}\ge (1-a)y$ for all $a<1$ sufficiently small and $y\in [0, 2a]$.
For any such $a$, choose $t$ so large that $c(t)<a$. For $x\in [c(t),a]$, it then follows
\begin{align}
\Prob\{S_{\lfloor x/c(t)}>t\}\ &\ge\  \Prob\{\max_{1\le k\le\lfloor x/c(t)\rfloor}\,\xi_{k}>t\}\ =\ 1-\exp(\lfloor x/c(t)\rfloor\log\Prob\{\xi\le t\})\nonumber\\
&\ge\ 1-\exp(-\lfloor x/c(t)\rfloor c(t))\ \ge\ (1-a)\lfloor x/c(t)\rfloor c(t)
\label{eq:relmax}
\end{align}
and thereupon
\begin{align*}
\int_{c(t)}^{a}&-\log \Prob\{\tau(t)\le \lfloor x/c(t)\rfloor\}\ {\rm d}x\ =\ \int_{c(t)}^{a} -\log \Prob\{S_{\lfloor x/c(t)\rfloor}>t\}\ {\rm d}x\\
&\le\ -(a-c(t))\log(1-a)\,-\,\sum_{k=1}^{\lfloor a/c(t)\rfloor}\int_{kc(t)}^{(k+1)c(t)}\log\big(\lfloor x/c(t)\rfloor c(t)\big)\ {\rm d}x\\
&\le\ -\log(1-a)\,-\,c(t)\sum_{k=1}^{\lfloor a/c(t)\rfloor}\log(kc(t)).
\end{align*}
This in combination with $\lim_{t\to\infty}c(t)\sum_{k=1}^{\lfloor a/c(t)\rfloor}\log(kc(t))=\int_0^a \log y\ {\rm d}y$ allows us to conclude
$$ \lim_{a\searrow 0}\limsup_{t\to\infty}\int_{c(t)}^{a} \big(-\log \Prob\{\tau(t)\le \lfloor x/c(t)\rfloor\}\big){\rm d}x\ =\ 0, $$
which completes the proof of part (a).

\vspace{.2cm}
(b) We first argue that the law of $\xi$ belongs to the domain of attraction of a suitable stable law under the given tail assumption with $\alpha>1$. In fact, if $\alpha\in (1,2)$, or $\alpha\ge 2$ and ${\rm Var}\,\xi<\infty$ (automatically fulfilled if $\alpha>2$), then this has already been pointed out at the beginning of Section \ref{sect:weak convergence}, the stable law having index $\alpha\wedge 2$. For the remaining case $\alpha=2$ and ${\rm Var}\,\xi=\infty$, our assumption $\Prob\{\xi>t\}\sim t^{-2}\ell(t)$ as $t\to\infty$ ensures that $\mu_{2}(t):=\int_{[0,\,t]}y^{2}\,\Prob\{\xi\in {\rm d}y\}$ belongs to the de Haan class $\Pi$ with auxiliary function $\ell$, that is
$$ \lim_{t\to\infty}\frac{\mu_{2}(ht)-\mu_{2}(t)}{\ell(t)}\ =\ \log h\quad\text{for all }h>0, $$
and it is known that any such function is slowly varying at $\infty$. But this is indeed a necessary and sufficient condition for the law of $\xi$ to be in the non-normal domain of attraction of a normal distribution. Having thus verified that the law of $\xi$ is attracted by a stable law, Lemma \ref{lem:aux} tells us that it suffices to prove
$$ \lim_{t\to\infty}\frac{-\log\prod_{n=1}^{\lfloor t/\mu\rfloor}\Prob\{S_{n}> t\}}{t\log t}\ =\ \frac{\alpha-1}{\mu}.$$
To this end, we make use of the following large deviation result that follows directly from Thm.\,1 in \cite{Nagaev:1982} or Thm.\,3.3 in \cite{Cline+Hsing:2023}, namely
\begin{equation}\label{eq:Cline+Hsing LDP}
\lim_{n\to\infty}\sup_{t\ge \delta n}\Big|\frac{\Prob\{S_{n}-n\mu>t\}}{n\Prob\{\xi>t\}}-1\Big|\ =\ 0\quad\text{for all }\delta>0.
\end{equation}
For any fixed $\delta$, it entails
$$ \sum_{n=1}^{\lfloor t/(\mu+\delta)\rfloor}\log \Prob\{S_{n}-\mu n>t-\mu n\}\ \sim\ \sum_{n=1}^{\lfloor t/(\mu+\delta)\rfloor}\big(\log n+\log \Prob\{\xi>t-\mu n\}\big)\quad\text{as }t\to\infty. $$
Observing
\begin{gather*}
\sum_{n=1}^{\lfloor t/(\mu+\delta)\rfloor}\log n\ \sim\ \frac{t\log t}{\mu+\delta}\quad\text{as }t\to\infty
\intertext{and that, by the given tail assumption on the law of $\xi$,}
\lim_{t\to\infty}\frac{-\log \Prob\{\xi>t\}}{\log t}\ =\ \alpha,
\end{gather*}
we infer, for any $\eps\in (0,\alpha)$, all sufficiently large $t$, and all positive integers $n\le\lfloor t/(\mu+\delta)\rfloor$, the inequality
$$ (\alpha-\eps)\log (t-\mu n)\ \le\ -\log \Prob\{\xi>t-\mu n\}\ \le\ (\alpha+\eps)\log (t-\mu n). $$
Since
\begin{align*}
\sum_{n=1}^{\lfloor t/(\mu+\delta)\rfloor}&\log (t-\mu n)=\lfloor  t/(\mu+\delta)\rfloor \log t+\sum_{n=1}^{\lfloor t/(\mu+\delta)\rfloor}\log (1-\mu n/t)\\
&=\ (t/(\mu+\delta))\log t+t\int_{0}^{1/(\mu+\delta)}\log (1-\mu x){\rm d}x+ o(t)\quad\text{as }t\to\infty,
\end{align*}
we arrive at the conclusion
$$ \lim_{t\to\infty}\frac{1}{t\log t}\sum_{n=1}^{\lfloor t/(\mu+\delta\rfloor)}(-\log \Prob\{S_{n}-\mu n>t-\mu n\})\ =\ \frac{\alpha-1}{\mu+\delta}. $$
To complete the proof of part (b), we still need to verify that
$$ \lim_{\delta \searrow 0}\limsup_{t\to\infty}\frac{1}{t\log t}\sum_{n=\lfloor t/(\mu+\delta\rfloor)+1}^{\lfloor t/\mu\rfloor}(-\log \Prob\{S_{n}>t\})\ =\ 0. $$
We can argue as in \eqref{eq:relmax} to infer
\begin{align*}
-\log \Prob\{S_{\lfloor t/(\mu+\delta)\rfloor}>t\}\ &\le\ -\log\Prob\{\max_{1\le k\le \lfloor t/(\mu+\delta)\rfloor}\,\xi_{k}>t\}\\
&=\ -\log (1-\exp(\lfloor t/(\mu+\delta)\rfloor\log\Prob\{\xi\le t\}))\\
&\sim\ -\log t-\log \Prob\{\xi>t\}\\
&\sim\ (\alpha-1)\log t
\end{align*}
as $t\to\infty$, which in combination with
$$ \sum_{n=\lfloor t/(\mu+\delta\rfloor)+1}^{\lfloor t/\mu\rfloor}(-\log \Prob\{S_{n}>t\})\ \le\ \big(\lfloor t/\mu\rfloor-\lfloor t/(\mu+\delta)\rfloor\big)(-\log \Prob\{S_{\lfloor t/(\mu+\delta)\rfloor}>t\}) $$
provides the desired result.
\end{proof}

\begin{proof}[Proof of Theorem \ref{thm:semi}]
Since, obviously, $\Erw\xi^{p}<\infty$ for all $p>0$, the law of $\xi$ belongs to the normal domain of attraction of a normal law. Again by Lemma \ref{lem:aux}, it suffices to prove
$$ \lim_{t\to\infty}\frac{-\log \prod_{n=1}^{\lfloor t/\mu\rfloor}\Prob\{S_{n}\le t\}}{t^\alpha \ell(t)}\ =\ \frac{1}{\mu(\alpha+1)}. $$
By Theorem 2.1 in \cite{Borovkov+Mogulskii:2006},
$$ \lim_{n\to\infty}\sup_{t\ge n^{1/(2-\alpha)}\ell_{1}(n)f(n)}\bigg|\frac{\log \Prob\{S_{n}-\mu n>t\}}{\log \Prob\{\xi>t\}}-1\bigg|\ =\ 0, $$
where $\ell_{1}$ varies regularly at $\infty$ (its explicit form is of no importance for the present proof) and $f$ denotes a positive function diverging to $\infty$ as $n\to\infty$. Since $\mu n+n^{1/(2-\alpha)}\ell_{1}(n)f(n)\le t$ holds for any positive integer $n\le \lfloor t/\mu\rfloor$ when choosing $t$ sufficiently large, we infer
\begin{gather*}
\sum_{n=1}^{\lfloor t/\mu\rfloor}\big(-\log \Prob\{S_{n}-\mu n>t-\mu n\}\big)\ \sim\ \sum_{n=1}^{\lfloor t/\mu\rfloor} (t-\mu n)^\alpha \ell(t-\mu n)\ \sim\ \int_{0}^{t/\mu}(t-\mu x)^\alpha\ell(t-\mu x){\rm d}x\\
=\ \frac{1}{\mu}\int_{0}^t x^\alpha \ell(x)\ {\rm d}x\ \sim\ \frac{t^{\alpha+1}\ell(t)}{\mu(\alpha+1)}
\end{gather*}
as $t\to\infty$ and thus the above limit assertion.
\end{proof}

\section{Proof of Theorem \ref{thm:SLLN maxima}}

The obvious duality relation
\begin{equation}\label{eq:duality max and passage time}
\frac{M_{\wh \tau(t)-1}}{\wh \tau(t)}<\frac{t}{\wh \tau(t)}\le \frac{M_{\wh \tau(t)}}{\wh \tau(t)}\quad \text{a.s.},
\end{equation}
valid for all $t\ge 0$, shows that any law of large numbers type result for the decoupled maxima $M_{n}$ also yields a limit result for $\wh{\tau}(t)$ without further ado. Our proof of Theorem \ref{thm:SLLN maxima} therefore only deals with the assertions for the maxima. The following one-sided version of the Hsu-Robbins-Erd\"os theorem (see, for instance, Thm.\,6.11.2 in \cite{Gut:2013}) and two subsequent lemmata serve as auxiliary results.

\begin{assertion}\label{prop:one-sided Hsu-Robbins}
Let $(S_{n})_{n\ge 1}$ be a standard random walk with drift $\mu=0$. Then
\begin{align*}
\Sigma(\eps)\,:=\,\sum_{n\ge 1}\Prob\{S_{n}\ge\eps n\}\,<\,\infty\quad\text{for some/all }\eps>0
\end{align*}
holds if, and only if, $\Erw(\xi^{+})^{2}<\infty$.
\end{assertion}

\begin{proof}
Putting $S_{n}(\eps):=\eps n-S_{n}$, we see that $\Sigma(\eps)=\sum_{n\ge 1}\Prob\{S_{n}(\eps)\le 0\}$ equals the renewal function at $0$ of the random walk
$(S_{n}(\eps))_{n\ge 1}$. It is a well-known fact from renewal theory (see, for instance, p.\,94 in \cite{Gut:2009}) that under the assumption that $\mu$ is finite, this function is finite if, and only if, $\Erw(S_{1}(\eps)^{-})^{2}=\Erw((\xi-\eps)^{+})^{2}<\infty$.
\end{proof}

We put $\ovl{F}=1-F$ for a distribution function $F$.

\begin{lemma}\label{lem:slowly varying tails}
Let $\xi$ be a nonnegative random variable with distribution function $F$ that satisfies
\begin{equation}\label{eq:loglog tail condition}
\lim_{t\to\infty}\frac{t^{2}\ovl{F}(t)}{\log\log t}\ =\ 0.
\end{equation}
Then there exists a distribution function $G\le F$ that satisfies \eqref{eq:loglog tail condition} as well and is such that $t^{2}\ovl{G}(t)$ is slowly varying at infinity. 
The function $G$ may further be chosen subject to $\int_{0}^{\infty}\ovl{G}(x)\,{\rm d}x\le\mu+\frac{1}{n}$ for arbitrary $n\in\N$, where $\mu=\Erw\xi$.
\end{lemma}

\begin{proof}
Define
$$ a_{0}\,:=\,\inf\{t\ge \eee:t^{-2}\log\log t\text{ is decreasing}\} $$
and recursively
$$ a_{n}\,:=\,\inf\bigg\{t\ge na_{n-1}:\frac{s^{2}\,\ovl{F}(s)}{\log\log s}\le\frac{1}{n}\text{ for all }s\ge t\bigg\} $$
for $n\in\N$. Then
$$ G(t)\ :=\ 1-\1_{[0,a_{0})}(t)-\frac{\log\log t}{t^{2}}\sum_{n\ge 1}\frac{1}{n}\1_{[a_{n-1},a_{n})}(t),\quad t\ge 0 $$
is a distribution function on $[0,\infty)$ that satisfies the tail condition \eqref{eq:loglog tail condition} and is also bounded by $F$. Moreover, it can be verified by using $a_{n}\ge n
a_{n-1}$ that $t^{2}\ovl{G}(t)$ is slowly varying at infinity.
Defining
$$ G_{m}(t)\ =\ \frac{G(m)}{F(m)}\,F(t)\1_{[0,m)}(t)\ +\ G(t)\,\1_{[m,\infty)}(t) $$
for $m\in\N$, one can further readily check that the $G_{m}$ are distribution functions that also have the properties asserted for $G$ and that $\lim_{m\to\infty}\int_{0}^{\infty}\ovl{G}_{m}(x)\,{\rm d}x=\mu$.
\end{proof}

\begin{lemma}\label{lem:BC-type lemma}
Suppose that
\begin{gather}\label{eq:slowly varying tail condition}
t^{2}\,\Prob\{\xi>t\}\text{ is slowly varying at infinity}
\shortintertext{and}
\lim_{t\to\infty}\frac{t^{2}\,\Prob\{\xi>t\}}{\log\log t}\ =\ 0.\label{eq:lower loglog condition}
\end{gather}
Put $l_{n}=n\log n$ for $n\in\N$. Then
\begin{equation*}
\sum_{n\ge 1}\Prob\{M_{b^{l_{n}}}>cb^{l_{n+1}}\}\ <\ \infty
\end{equation*}
for any $c>0$ and integer $b\ge 2$, and
\begin{equation*}
\sum_{n\ge 1}\Prob\bigg\{\max_{b^{l_{n}}<k\le b^{l_{n+1}}}\wh{S}_{k}\le cb^{l_{n+1}}\bigg\}\ =\ \infty
\end{equation*}
for any $c>\mu$ and integer $b\ge 2$, where $\mu=\Erw\xi<\infty$.
\end{lemma}

\begin{proof}
Fixing an arbitrary $c>0$ and $\eps\in (0,1)$, we have for all sufficiently large $n$
\begin{align*}
\Prob\{M_{b^{l_{n}}}>2cb^{l_{n+1}}\}\ &\le\ \sum_{1\le k\le b^{l_{n}}}\Prob\{S_{k}>2cb^{l_{n+1}}\}\ \le\ b^{l_{n}}\,\Prob\{S_{b^{l_{n}}}>2cb^{l_{n+1}}\}\\
&\le\ b^{l_{n}}\,\Prob\{S_{b^{l_{n}}}-\mu b^{l_{n}}>cb^{l_{n+1}}\}\ \le\ (1+\eps)b^{2l_{n}}\,\Prob\{\xi>cb^{l_{n+1}}\}\\
&\le\ c^{-1}(1+\eps)^{2}\,b^{2(l_{n}-l_{n+1})}\log\log cb^{l_{n+1}},
\end{align*}
where \eqref{eq:Cline+Hsing LDP} has been utilized for the fourth inequality. The first assertion now follows because
$$ 
b^{2(l_{n}-l_{n+1})}\log\log cb^{l_{n+1}}\ \sim\ b^{-2}n^{-2\log b}\log n\quad\text{as }n\to\infty $$
for any $c>0$ and integer $b\ge 2$. For the second assertion, we fix an arbitrary $c>\mu$ and put $c'=c-\mu$. Then we obtain
\begin{align*}
\Prob\bigg\{\max_{b^{l_{n}}<k\le b^{l_{n+1}}
}\wh{S}_{k}\le cb^{l_{n+1}}\bigg\}\ &=\ \exp\Bigg(\sum_{b^{l_{n}}<k\le b^{l_{n+1}}}\log\big(1-\Prob\big\{\wh{S}_{k}>cb^{l_{n+1}}\big\}\big)\Bigg)\\
&\ge\ \exp\Big(-(1+\eps)(b^{l_{n+1}}-b^{l_{n}})\,\Prob\big\{\wh{S}_{b^{l_{n+1}}}>cb^{l_{n+1}}\big\}\Big)\\
&=\ \exp\Big(-(1+\eps)(b^{l_{n+1}}-b^{l_{n}})\,\Prob\big\{\wh{S}_{b^{l_{n+1}}}-\mu b^{l_{n+1}}>c'b^{l_{n+1}}\big\}\Big)\\
&\ge\ \exp\Big(-(1+\eps)^{2}(b^{l_{n+1}}-b^{l_{n}})b^{l_{n+1}}\,\Prob\big\{\xi>c'b^{l_{n+1}}\big\}\Big)\\
&\ge\ \exp\Big(-(1+\eps)^{2}b^{2l_{n+1}}\,\Prob\big\{\xi>c'b^{l_{n+1}}\big\}\Big)\\
&\ge\ \exp\Big(-(1+\eps)^{2}\eps_{n}\,\log\log c'b^{l_{n+1}}\Big)
\end{align*}
for all sufficiently large $n$ and suitable $\eps_{n}\to 0$. Hence, using $\log\log c'b^{l_{n+1}}\sim\log n$, we see that
$$ \exp\Big(-(1+\eps)^{2}\eps_{n}\,\log\log c'b^{l_{n+1}}\Big)\ \asymp\ n^{-\eps _{n}(1+\eps)^{2}}, $$
which gives the desired result. Here, $a_{n}\asymp b_{n}$ means that $a_{n}/b_{n}$ is bounded and bounded away from $0$.
\end{proof}

\begin{proof}[Proof of Theorem \ref{thm:SLLN maxima}]
In view of the duality relation \eqref{eq:duality max and passage time}, it suffices to prove the assertions for $(M_{n})_{n\ge 1}$ as already mentioned.

\vspace{.2cm}
(a) For all $\eps\in (0,\mu)$,
$$ \{|M_{n}-\mu n|>\eps n~\text{i.o.}\}\ \subseteq\ \{|\wh{S}_{n}-\mu n|>\eps n~\text{i.o.}\}, $$ where ``\hspace{1pt}i.o.'' is the usual abbreviation for ``\hspace{1pt}infinitely often''. Since $\Erw\xi^{2}<\infty$, Prop.~\ref{prop:one-sided Hsu-Robbins} implies
$$ \sum_{n\ge 1}\Prob\{|\wh{S}_{n}-\mu n|>\eps n\}\ =\ \sum_{n\ge 1}\Prob\{|S_{n}-\mu n|>\eps n\}\ <\ \infty $$
and thus $\Prob\{|M_{n}-\mu n|>\eps n~\text{i.o.}\}=\Prob\{|\wh{S}_{n}-\mu n|>\eps n~\text{i.o.}\}=0$. This proves the first limit relation in \eqref{eq:slln}.

\vspace{.2cm}
(b) Assume next that $\Erw\xi^{2}=\infty$, thus $\Erw((\xi-\mu)^{-})^{2}<\infty=\Erw((\xi-\mu)^{+})^{2}$, for $\xi$ is nonnegative. Then, by another appeal to Prop.~\ref{prop:one-sided Hsu-Robbins},\begin{gather*}
\sum_{n\ge 1}\Prob\{\wh{S}_{n}>(\mu+\eps)n\}\ =\ \sum_{n\ge 1}\Prob\{S_{n}>(\mu+\eps)n\}\ =\ \infty\quad\text{for all }\eps>0,
\shortintertext{whereas}
\sum_{n\ge 1}\Prob\{\wh{S}_{n}<(\mu-\eps)n\}\ =\ \sum_{n\ge 1}\Prob\{S_{n}<(\mu-\eps)n\}\ <\ \infty\quad\text{for all }\eps>0.
\end{gather*}
Consequently, $\limsup_{n\to\infty} n^{-1}M_{n}=\limsup_{n\to\infty} n^{-1}\wh{S}_{n}=\infty$ a.s.~by the converse part of the Borel-Cantelli lemma and
\begin{equation}\label{eq:later}
\liminf_{n\to\infty} n^{-1}M_{n}\ \ge\ \liminf_{n\to\infty}n^{-1}\wh{S}_{n}\ \ge\ \mu\quad\text{a.s.}
\end{equation}
by the direct part of the Borel-Cantelli lemma (relation \eqref{eq:later} will be used later).

\vspace{.1cm}
In order to show $\lim_{n\to\infty}n^{-1}M_{n}=\infty$ a.s.~under the additional assumption
\begin{equation}\label{eq:upper loglog condition}
\lim_{t\to\infty}\frac{t^{2}\,\Prob\{\xi>t\}}{\log\log t}\ =\ \infty,
\end{equation}
we first note that
\begin{equation}\label{eq:lower tail equivalence}
\liminf_{n\to\infty}\frac{\Prob\{S_{n}>cn\}}{n\,\Prob\{\xi>cn\}}\ \ge\ 1
\end{equation}
for any $c>0$. The latter follows from
$$ \Prob\{S_{n}>cn\}\ \ge\ \Prob\{\max_{1\le k\le n}\,\xi_k>cn\}\ =\ 1-F(cn)^{n}\ \sim\ n\,\Prob\{\xi>cn\}\quad \text{as }n\to\infty, $$
where the limit relation is a consequence of  $\lim_{n\to\infty}n\,\Prob\{\xi>cn\}=0$ 
(as $\mu=\Erw\xi<\infty$).


\vspace{.1cm}
Fixing any $\eps\in (0,1)$ and $c>0$, \eqref{eq:lower tail equivalence} provides us with
$$ \Prob\{S_{n}>cn\}\ \ge\ (1-\eps)n\,\Prob\{\xi>cn\} $$
for all sufficiently large $n$. Consequently, putting $\ovl{M}_{n}(b):=\max_{b^{n-1}\le
k<b^{n}}k^{-1}\wh{S}_{k}$,
\begin{align*}
\Prob\big\{\ovl{M}_{n}(b)\le c\big\}\
&=\ \exp\Bigg(\sum_{b^{n-1}\le k<b^{n}}\log\big(1-\Prob\{S_{k}>ck\}\big)\Bigg)\\
&\le\ \exp\Bigg(-\sum_{b^{n-1}\le k<b^{n}}\Prob\{S_{k}>ck\}\Bigg)\\
&\le\ \exp\Bigg(-(1-\eps)\sum_{b^{n-1}\le k<b^{n}}k\,\Prob\{\xi>ck\}\Bigg)\\
&\le\ \exp\big(-(1-\eps)(b-1)b^{2n-2}\,\Prob\{\xi>cb^{n}\}\big).
\end{align*}
for any integer $b>1$ and sufficiently large $n$. Now use \eqref{eq:upper loglog condition}, giving $\lim_{n\to\infty}b^{2n}\,\Prob\{\xi>cb^{n}\}/\log n=\infty$, to infer
\begin{align*}
\sum_{n\ge 1}\Prob\big\{\ovl{M}_{n}(b)\le c\big\}\ <\ \infty
\end{align*}
and thus $\Prob\big\{\ovl{M}_{n}(b)\le c\text{ i.o.}\big\}=0$ for any integer $b>1$ and $c>0$ by another appeal to the Borel-Cantelli lemma. We arrive at the desired conclusion (first half of \eqref{eq:slln3}) because
$$ \lim_{n\to\infty}\frac{M_{n}}{n}\ =\ \lim_{n\to\infty}\ovl{M}_{n}(b)\ =\ \infty\quad\text{a.s.} $$

In view of \eqref{eq:later} it remains to show $\liminf_{n\to\infty}n^{-1}M_{n}\le\mu$ a.s.~if $\Erw\xi^{2}=\infty$ and
\eqref{eq:lower loglog condition} holds. W.l.o.g.~we make the additional assumption that the law of $\xi$ also satisfies \eqref{eq:slowly varying tail condition}. Otherwise, Lemma \ref{lem:slowly varying tails} provides the existence of a coupling $(\xi_{n,k},\xi_{n,k}')_{n,k\ge 1}$ of i.i.d.~random pairs with generic copy $(\xi,\xi')$ such that $\xi'\ge\xi$ a.s., $\Erw\xi'\in (\mu,\mu+\eps)$ for arbitrarily fixed $\eps>0$, and $t^{2}\,\Prob\{\xi'>t\}$ satisfies both \eqref{eq:slowly varying tail condition} and \eqref{eq:lower loglog condition}. Putting $\wh{S}_{n}=\sum_{k=1}^{n}\xi_{n,k},\,\wh{S}_{n}'=\sum_{k=1}^{n}\xi_{n,k}'$ and $M_{n}'=\max_{1\le k\le n}\wh{S}_{k}'$, we then obviously have $M_{n}\le M_{n}'$ a.s. Hence, by proving the assertion for the $M_{n}'$, i.e.,~$\liminf_{n\to\infty}n^{-1}M_{n}'\le\Erw\xi'\le\mu+\eps$ a.s., we also get the result for $M_{n}$.

\vspace{.1cm}
If the law of $\xi$ satisfies \eqref{eq:slowly varying tail condition} and \eqref{eq:lower loglog condition}, we can invoke Lemma \ref{lem:BC-type lemma} and the Borel-Cantelli lemma to infer
\begin{gather*}
\Prob\left\{M_{b^{l_{n}}}>cb^{l_{n+1}}\text{ i.o.}\right\}\ =\ 0
\shortintertext{and}
\Prob\bigg\{\max_{b^{l_{n}}<k\le b^{l_{n+1}}}\wh{S}_{k}\le cb^{l_{n+1}}\text{ i.o.}\bigg\}\ =\ 1
\end{gather*}
for any $c>\mu$. When combined, this yields
$$ \Prob\left\{M_{n}\le cn\text{ i.o.}\right\}\ =\ \Prob\left\{M_{b^{l_{n}}}\le cb^{l_{n}}\text{ i.o.}\right\}\ =\ 1 $$
for any $c>\mu$ and thus the desired result.

\vspace{.2cm}
(c) If $\Erw\xi=\infty$, a simple truncation argument provides $\lim_{n\to\infty} n^{-1}M_{n}=\infty$ a.s. Namely, let again $\wh{S}_{n}=\sum_{k=1}^{n}\xi_{n,k}$ for $n\in\N$ and consider the decoupled random walk $(\wh{S}_{n}(b))_{n\ge 1}$ with increments $\xi_{n,k}\wedge b$ for $b>0$ and associated maxima $M_{n}(b)=\max_{1\le k\le n}\wh{S}_{k}(b)$. Plainly, $M_{n}\ge M_{n}(b)$ a.s.~for all $n\in\N$ and $b>0$, and since $\Erw(\xi\wedge b)^{2}<\infty$, we infer with the help of part (a)
$$ \liminf_{n\to\infty}\frac{M_{n}}{n}\ \ge\ \lim_{n\to\infty}\frac{M_{n}(b)}{n}\ =\ \Erw(\xi\wedge b)\quad\text{a.s.} $$
for any $b>0$ and thereupon the assertion because $\lim_{b\to\infty}\Erw(\xi\wedge b)=\infty$.

\vspace{.2cm}
(d) Let $\tau(t)=\inf\{n\ge 1:S_{n}>t\}$ be the level-$t$ first passage time for 
$(S_{n})_{n\ge 1}$. It is well-known from standard renewal theory (see, for instance, the proof of Theorem 2.5.1 on p.\,58 in \cite{Gut:2009}) that the family $\{t^{-1}\tau(t):t\ge t_{0}\}$ is uniformly integrable for any $t_{0}>0$. Furthermore,
\begin{align*}
\Prob\{\wh{\tau}(t)>n\}\ =\ \prod_{k=1}^{n}\Prob\{S_{k}\le t\}\ \le\ \Prob\{S_{n}\le t\}\ =\ \Prob\{\tau(t)>n\}
\end{align*}
for all $n\in\N$ and $t\ge 0$. This shows that the distribution tails of $\wh{\tau}(t)$ are dominated by the distribution tails of $\tau(t)$ for each $t$, and this implies the uniform integrability of the family
$\{t^{-1}\wh\tau(t):t\ge t_{0}\}$.
\end{proof}

\noindent \textbf{Funding.} Gerold Alsmeyer and Zakhar Kabluchko acknowledge support by the German Research Foundation (DFG) under Germany’s Excellence Strategy EXC 2044 – 390685587, Mathematics M\"unster: Dynamics--Geometry--Structure. 

\noindent \textbf{Acknowledgment.} A part of this work was done while Alexander Iksanov was visiting M\"{u}nster in March-April 2023 as a M\"{u}nster research fellow. Grateful acknowledgment is made for financial support and hospitality.


\begin{thebibliography}{30}

\bibitem{Baum+Katz:1965} L.E. Baum and M. Katz. \textit{Convergence rates in the law of large numbers}. Trans. Amer. Math. Soc. \textbf{120} (1965), 108--123.

\bibitem{Billingsley:1999} P. Billingsley. \textit{Convergence of probability measures}. 2nd edition, Wiley, 1999.

\bibitem{Borovkov+Mogulskii:2006} A.~A. Borovkov and A.~A. Mogul'skii. \textit{Integro-local and integral theorems for sums of random variables with semiexponential distributions}. Siberian Math. J. \textbf{47} (2006), 990--1026.

\bibitem{Cline+Hsing:2023} D.B.H. Cline and T. Hsing. \textit{Large deviation probabilities for sums of random variables with heavy or subexponential tails}. Preprint (2022) available at {\tt https://arxiv.org/pdf/ 2211.16340.pdf}

\bibitem{Hough+Krishnapur+Peres+Virag:2009} J. Ben Hough, M. Krishnapur, Y. Peres and B. Vir\`{a}g. \textit{Zeros of Gaussian analytic functions and determinantal point processes}. American Mathematical Society, Vol. \textbf{51}, 2009.

\bibitem{Bingham+Goldie+Teugels:1989} N.H. Bingham, C.M. Goldie and J.L. Teugels. \textit{Regular variation}. Cambridge University Press, 1989.

\bibitem{Durrett:2010} R. Durrett. \textit{Probability: theory and examples}. 4th edition, Cambridge University Press, 2010.

\bibitem{Feller:1949} W. Feller, \textit{Fluctuation theory of recurrent events}. Trans. Amer. Math. Soc. \textbf{67} (1949), 98--119.

\bibitem{Fenzl+Lambert:2021} M. Fenzl and G. Lambert. \textit{Precise deviations for disk counting statistics of invariant determinantal processes}. Int. Math. Res. Not. \textbf{2022} (2022), 7420-–7494.

\bibitem{Gut:2009} A. Gut. \textit{Stopped random walks: {L}imit theorems and applications}, 2nd edition. Springer, 2009.

\bibitem{Gut:2013} A. Gut. \textit{Probability: {A} graduate course}, 2nd edition. Springer, 2013.

\bibitem{Hollander:2000} F. den Hollander, \textit{Large deviations}. American Mathematical Society, 2000.

\bibitem{Iksanov+Kabluchko+Kotelnikova:2022} A. Iksanov, Z. Kabluchko and V. Kotelnikova. \textit{A functional limit theorem for nested Karlin’s occupancy scheme generated by discrete Weibull-like distributions}. J. Math. Anal. Appl. \textbf{507} (2022), 125798.

\bibitem{Iksanov+Marynych+Meiners:2016} A. Iksanov, A. Marynych and M. Meiners. \textit{Moment convergence of first-passage times in renewal theory}. Statist. Probab. Letters. \textbf{119} (2016), 134--143.

\bibitem{Kostlan:1992} E. Kostlan, \textit{On the spectra of Gaussian matrices}. Lin. Algebra Appl. \textbf{162–164} (1992), 385--388.

\bibitem{Nagaev:1982} S.V. Nagaev, \textit{On the asymptotic behavior of one-sided large deviation probabilities}. Theor. Prob. Appl. \textbf{26} (1982), 362--366.

\bibitem{Whitt:2002} W. Whitt, \textit{Stochastic-process limits. An introduction to stochastic-process limits and their application to queues}. Springer, 2002.

\bibitem{Yamazato:1978} M. Yamazato, \textit{Unimodality of infinitely divisible distribution functions of class $L$}. Ann. Probab. \textbf{6} (1978), 523--531.

\end{thebibliography}
\end{document}